\documentclass[12pt,reqno]{amsart} 

\usepackage{float} 

\usepackage{epsfig,amssymb,amsmath,version}
\usepackage{amssymb,version,graphicx,fancybox,mathrsfs}

\usepackage{url,hyperref}
\usepackage{subfigure}
\usepackage{color}
\usepackage{stmaryrd}
\usepackage{multirow}
\usepackage{booktabs,siunitx}
\usepackage{booktabs}
\usepackage{multicol,epstopdf}
\usepackage{xypic}
\usepackage[]{graphicx}
\usepackage[all]{xy}
\usepackage{tikz,ifthen}
\usetikzlibrary{positioning, math, decorations.markings, arrows.meta}
\usetikzlibrary{calc,shapes,cd}
\usetikzlibrary{knots} 
\usepgfmodule{decorations}
\allowdisplaybreaks
\tikzset{knotarrow/.pic={ \draw[edge, <-] (0,0) -- +(-.001,0);}}

\tikzset{edge/.style={line width=0.8}}
\tikzset{wall/.style={very thick}}

\tikzset{->-/.style n args={2}{decoration={markings, mark=at position #1 with {\arrow{#2}}}, postaction={decorate}}} 

\ExplSyntaxOn
\cs_new_eq:NN \ifstreqF \str_if_eq:nnF
\cs_new_eq:NN \ifstreqTF \str_if_eq:nnTF
\ExplSyntaxOff

\tikzset{-o-/.code 2 args={\ifstreqF{#2}{} 
{\ifstreqTF{#2}{>}
   {\pgfkeysalso{decoration={markings,mark=at position #1 with {\arrow[scale=0.8]{#2}}}
                    ,postaction={decorate}}
    }
   {\ifstreqTF{#2}{<}
       {\pgfkeysalso{decoration={markings,mark=at position #1 with {\arrow[scale=0.8]{#2}}}
                    ,postaction={decorate}}
        }
       {\pgfkeysalso{decoration={markings,
                    mark=at position #1 with
                    {\draw[black, fill={#2}] circle[radius=2pt];}}
                    ,postaction={decorate}}
        }
     }
  }}}

\allowdisplaybreaks[4]

\newtheorem{theorem}{Theorem}[section]
\newtheorem{lemma}[theorem]{Lemma}

\newtheorem{corollary}[theorem]{Corollary}

\newtheorem{rem}[theorem]{Remark}
\newtheorem{conjecture}[theorem]{Conjecture}

\definecolor{ligreen}{rgb}{0.0, 0.3, 0.0}

\definecolor{darkblue}{rgb}{0.0, 0.0, 0.55}

\definecolor{anti-flashwhite}{rgb}{0.55, 0.57, 0.68}

\def\cP{\mathcal P}

\def\pr{\mathrm{pr}}

\def\ot{\otimes}
\def\buu{{\mathbf u}}

\def\cR{\mathcal R}

\def\MN {(M,\mathcal{N})}
\def\cN {\mathcal{N}}

\newcommand{\beq}{\begin{equation}}
\newcommand{\eeq}{\end{equation}}

\graphicspath{{./Figures/} } 

\begin{document}

\title{Stated $SL_n$-skein modules, roots of unity, and TQFT}

\author[Zhihao Wang]{Zhihao Wang}
\address{Zhihao Wang, School of Physical and Mathematical Sciences, Nanyang Technological University, 21 Nanyang Link Singapore 637371}
\email{ZHIHAO003@e.ntu.edu.sg}
\address{University of Groningen, Bernoulli Institute, 9700 AK Groningen, The Netherlands}
\email{wang.zhihao@rug.nl}

\keywords{Skein theory, $SL_n$, roots of unity, spin structure, splitting map, TQFT}

 \maketitle


\begin{abstract}

For a pb surface $\Sigma$, two positive integers $m,n$ with $m\mid n$, and two invertible elements $v,\epsilon$ in a commutative domain $R$ satisfying $\epsilon^{2m} = 1$,  
we construct an $R$-linear isomorphism between the stated $SL_n$-skein algebras $S_n(\Sigma,v)$ and $S_n(\Sigma,\epsilon v)$, which restricts to an algebra isomorphism between subalgebras of $S_n(\Sigma,v)$ and $S_n(\Sigma,\epsilon v)$.  
Using this linear isomorphism, we prove that the splitting map $\Theta_{c}:S_n(\Sigma,v)\rightarrow S_n(\text{Cut}_c(\Sigma),v)$ for the pb surface $\Sigma$ and the ideal arc $c$ is injective when $v^{2m} = 1$ and $m\mid n$.  

We generalize Barrett's work \cite{barrett1999skein} to the $SL_n$-skein space and the stated $SL_n$-skein space.  
As an application, we prove that the splitting map for marked 3-manifolds is always injective when the quantum parameter $v=-1$.  

Let $(M,\mathcal{N})$ be a connected marked 3-manifold with $\mathcal{N}\neq\emptyset$, and let $(M,\mathcal{N}')$ be obtained from $(M,\mathcal{N})$ by adding an extra marking. When $v^4 =1$, we prove that the $R$-linear map from $S_n(M,\mathcal{N},v)$ to $S_n(M,\mathcal{N}',v)$ induced by the embedding $(M,\mathcal{N})\rightarrow (M,\mathcal{N}')$ is injective and  
\[
S_n(M,\mathcal{N}',v) \simeq S_n(M,\mathcal{N},v)\otimes_{R} O_{q_v}(SL_n),
\]
where $O_{q_v}(SL_n)$ is the quantization of the coordinate ring of $SL_n$. This result confirms that the splitting map for  
$S_n(M,\mathcal{N},v)$ is always injective.  

We formulate the stated $SL_n$-TQFT theory, which generalizes the stated $SL_2$-TQFT theory developed in \cite{CL2022TQFT}.



\end{abstract}

\tableofcontents{}

\def \M {M,\mathcal{N}}

\section{Introduction}
In this paper, we will work with a commutative domain $R$ with an invertible element $v$, working as the quantum parameter. When we talk about a 3-manifold, we always mean an oriented  3-manifold with  a chosen
Riemannian metric.

 A {\bf marked  3-manifold} is  a pair $\MN$, where $M$ is a   3-manifold, and $\cN$ is a submanifold of $\partial M$ consisting of oriented open intervals such that there is no intersection between the closure of any two components of $\cN$. 
  Note that we allow $\cN$ to be empty.

  For a marked 3-manifold $\MN$ and a positive integer $n$, L{\^e} and Sikora defined the stated $SL_n$-skein module $S_n(\M,v)$ of $\MN$ \cite{le2021stated}, which is
  an $R$-module and quantized by the quantum parameter $v$, please refer to subsection \ref{sub111} for the definition of $S_n(\M,v)$. The classical limit of the stated $SL_n$-skein module is well-studied in \cite{S2001SLn,wang2023stated}.
  The stated $SL_n$-skein module is
   the stated version for the $SU_n$-skein theory defined by Sikora \cite{sikora2005skein}, and is also a generalization for the classical stated skein module \cite{le2018triangular}.

A {\bf punctured bordered surface} (or a {\bf pb surface}) is obtained from a compact surface by removing a finite set of points such that every boundary component is an open interval.
  
  When a marked 3-manifold $\MN$ is the thickening of a pb surface $\Sigma$, the stated $SL_n$-skein module $S_n(\M,v)$ has an algebra structure given by stacking the stated $n$-webs.
  So when $\MN$ is the thickening of a pb surface $\Sigma$, we will call the stated $SL_n$-skein module of $\MN$ as the stated  
    $SL_n$-skein algebra of $\Sigma$, denoted as $S_n(\Sigma,v)$.
  
  The stated $SL_n$-skein theory plays a key role in establishing the quantum trace map, an algebra embedding from the (stated) $SL_n$-skein algebra into a quantum torus \cite{bonahon2011quantum,le2018triangular,leY}. This map is instrumental in understanding the center and representation theory of the (stated) $SL_n$-skein algebra \cite{representation2,unicity,yu2023center}.

  For a pb surface $\Sigma$, positive  integers $m,n$ with $m\mid n$, and an element $\epsilon\in R$ with $\epsilon^{2m} = 1$, we construct an $R$-linear isomorphism $\varphi_{\epsilon}:S_n(\Sigma,v)\rightarrow S_n(\Sigma,\epsilon v)$, which restricts to an algebra isomorphism from a subalgebra  of $S_n(\Sigma,v)$ to a subalgebra of $S_n(\Sigma,\epsilon v)$.

\begin{theorem}\label{1}
	Let $\Sigma$ be a pb surface, let $m,n$ be two positive integers with $m\mid n$, and let $\epsilon\in R$ with $\epsilon^{2m} = 1$. Then there exists an $R$-linear isomorphism $\varphi_{\epsilon}:S_n(\Sigma,v)\rightarrow S_n(\Sigma,\epsilon v)$. 
\end{theorem}

For a 3-manifold $M$,
the $SL_n$-skein module $S_n(M,\emptyset,v)$ is quantized by $v^2$. So we will use $S_n(M,v^2)$ to denote $S_n(M,\emptyset,v)$.   Barrett built an $R$-linear isomorphism from $S_2(M,v^2)$ to $S_2(M,-v^2)$ using a spin structure of $M$ \cite{barrett1999skein}. This linear isomorphism is an algebra isomorphism when $M$ is the thickening of an oriented surface.  In this paper, we generalize Barrett's result to all $n$.

\begin{theorem}\label{2}
Let $M$ be a 3-manifold with a spin structure. For each positive interger $n$, there exists an $R$-linear isomorphism $\Phi_{n}:S_n(M,v^2)\rightarrow S_n(M,-v^2)$.
 In particular $\Phi_n$ is an algebra isomorphism when $M$ is the thickening of an oriented surface.
\end{theorem}

We also prove the stated version for Theorem \ref{2}.

\begin{theorem}
	Let $\MN$ be a marked 3-manifold, 
	and let $n$ be a positive integer.
	There exists an $R$-linear isomorphism $\Psi_n:S_n(\M,v)\rightarrow S_n(\M,-v)$. In particular, $\Psi_n$ is an algebra isomorphism when $\MN$ is the thickening of a pb surface and $\Psi_n(\alpha) = \alpha$ for any stated $n$-web $\alpha$ in $\MN$ without endpoints, please refer to subsection \ref{sub111} for the definition of the  stated $n$-web.

\end{theorem}

Costantino and L{\^e} developed the stated $SL_2$-TQFT theory \cite{CL2022TQFT}, constructing a symmetric monoidal functor from the category of decorated cobordisms to the Morita category. For precise definitions of these categories, see Section \ref{secTQFT}. The stated TQFT theory provides potential techniques for computing the dimension of skein modules of closed 3-manifolds and for proving Witten’s Finiteness Conjecture \cite{detcherry2021infinite,gunningham2023finiteness}. While the conjecture has been proven in \cite{gunningham2023finiteness}, its generalized form for compact 3-manifolds with boundary, as proposed in \cite{detcherry2021infinite}, remains open.  
For a closed 3-manifold $M$ and a positive integer $k$, let $M_k$ denote the marked 3-manifold obtained by removing $k$ open balls from $M$ and adding a marking to each newly created spherical boundary component. As an application of the stated $SL_2$-TQFT theory, the author proved that the dimension of the representation-reduced stated skein module of $M_k$ is one \cite{wang2023representation}.

We extend the stated $SL_2$-TQFT theory to $SL_n$. Specifically, we associate a marked surface with a stated $SL_n$-skein algebra and a decorated manifold with a stated $SL_n$-skein module, which is a bimodule over two stated $SL_n$-skein algebras. For a detailed discussion, see Section \ref{secTQFT}. This assignment defines a symmetric monoidal functor from the category of decorated cobordisms to the Morita category, as established in Theorem \ref{main}. Utilizing the stated $SL_n$-TQFT theory, we prove the following theorem.

\begin{theorem}\label{4}
	Let $\MN$ be a connected marked 3-manifold with $\cN \neq \emptyset$, and let $(M, \cN')$ be another marked 3-manifold obtained from $\MN$ by adding an extra marking $e'$. Suppose $\epsilon \in R$ satisfies $\epsilon^4 = 1$. Then:  
	
	\begin{enumerate}
		\item The $R$-linear map from $S_n(\M, \epsilon)$ to $S_n(M, \cN', \epsilon)$, induced by the embedding of $\MN$ into $(M, \cN')$, is injective.  
		\item There is an isomorphism  
		$$
		S_n(M, \cN', \epsilon) \simeq S_n(\M, \epsilon) \otimes_R O_{q_{\epsilon}}(SL_n),
		$$  
		where $O_{q_{\epsilon}}(SL_n)$ is defined in subsection \ref{bigon}.
	\end{enumerate}
\end{theorem}

When $\epsilon=1$,
Theorem \ref{4} is proved in \cite{wang2023stated} using a quite different and complicated technique.

L{\^e} and Sikora defined the splitting map for stated $SL_n$-skein modules and algebras, see subsections \ref{splitting} and \ref{sub27}, and  conjectured the injectivity of the splitting map for pb surfaces, Conjecture 7.12 in \cite{le2021stated} (or refer to Conjecture \ref{conj}). L{\^e} proved the injectivity of the splitting map for pb surfaces  when $n=2$ \cite{le2018triangular},  Higgins proved the case when $n=3$ \cite{higgins2020triangular}, L{\^e} and Sikora proved the Conjecture  when the pb surface is connected and has a non-empty boundary. 
The author proved the splitting map is injective for all marked 3-manifolds when the quantum parameter $v$ is $1$ \cite{wang2023stated}.

There are very few non-trivial affirmative examples of Conjecture \ref{conj} when the pb surface has no boundary and $n > 3$. The only non-trivial example known to the author is when the quantum parameter is $1$ \cite{wang2023stated}.

As  applications for the above Theorems, we prove the injectivity of the splitting map for some special cases.

\begin{theorem}
	Let $\Sigma$ be a pb surface with an ideal arc $c$, let $m,n$ be two positive intgers with $m\mid n$, and let $\epsilon\in R$ such that $\epsilon^{2m} = 1$. Then the splitting map $\Theta_{c}:S_n(\Sigma,\epsilon)\rightarrow S_n(\text{Cut}_c(\Sigma),\epsilon)$ is injective, where $\text{Cut}_c(\Sigma)$ is the pb surface obtained from $\Sigma$ by cutting along $c$, please refer to subsection \ref{sub27} for the detailed discussion.
\end{theorem}

\begin{theorem}
	Let $\MN$ be a connected marked  3-manifold with $\cN\neq\emptyset$,  let $D$ be a properly embedded disk in $M$, and let $e$ be an open oriented interval in $D$. Suppose $\epsilon\in R$ such that $\epsilon^4 = 1$. Then the splitting map
	$$\Theta_{(D,c)}:S_n(\M,\epsilon)\rightarrow S_n( \text{Cut}_{(D,e)}\MN,\epsilon)$$ is injective,   please refer to subsection \ref{splitting} for $\Theta_{(D,e)}$ and $\text{Cut}_{(D,e)}\MN$.
\end{theorem}

\begin{theorem}
	Let $\MN$ be marked 3-manifold,  let $D$ be a properly embedded disk in $M$, and let $e$ be an open oriented interval in $D$.
	Then the splitting map 
	$$\Theta_{(D,c)}:S_n(\M,-1)\rightarrow S_n( \text{Cut}_{(D,e)}\MN,-1)$$ is injective.
\end{theorem}
  

{\bf Acknowledgements}:
 The research is supported by the NTU research scholarship and
 the PhD scholarship from the University of Groningen.
We wish to thank the referee most warmly for numerous suggestions that have improved the exposition
of this paper.

\def\Si{\Sigma}

\section{Preliminary}

In this section, we will recall some definitions and results about stated $SL_n$-skein modules in \cite{le2021stated}.

\subsection{Stated $SL_n$-skein modules}\label{sub111}
In this subsection, we review the definition of the stated $SL_n$-skein module from \cite{le2021stated}, which generalizes the $SL_n$-skein module in \cite{sikora2005skein} to the stated $SL_n$-skein module.

 An {\bf $n$-web} $\alpha$ in a marked  3-manifold $\MN$ is a disjoint union of oriented closed paths and a directed finite  graph properly embedded in $M$. We also have the following requirements:

(1) $\alpha$ only contains $1$-valent or $n$-valent vertices. Each $n$-valent vertex is a source or a  sink. The set of one valent vertices is denoted as $\partial \alpha$, which are  called endpoints of $\alpha$.

 (2) Every edge $e$ of the graph is an embedded oriented  closed interval  in $M$.
 We can regard $e$ as an embedding from $[0,1]$ to $M$. Then $e(0)$ (resp. $e(1)$) is called the {\bf starting point} (resp. {\bf ending pointing}).

(3) $\alpha$ is equipped with a transversal framing. 

(4) The set of half-edges at each $n$-valent vertex is equipped with a  cyclic order. 

(5) $\partial \alpha$ is contained in $\cN$ and the framing at these endpoints is given by the velocity vector of $\cN$.

An $n$-web consisting of a single oriented framed circle is called a {\bf framed knot}.

For any two points $a,b\in\partial \alpha$, we say $a$ is higher than $b$ if they belong to a same component $e$ of $\cN$ and the direction of $e$ is going from $b$ to $a$.

The half edge of $\alpha$ containing the starting point (resp. ending point) is called the {\bf starting half edge} (resp. {\bf targeting half edge}).

A  {\bf state} of an $n$-web $\alpha$ is a map  $\partial \alpha\rightarrow \{1,2,\dots,n\}$. If there is  such a map for $\alpha$, we say $\alpha$ is {\bf a stated $n$-web.}

Recall that our ground ring is a commutative domain $R$ with an invertible element $v$. We set 
$q_v^{\frac{1}{2n}} = v$, and define the following constants:
\begin{align*}
c_{i,v}= (-q_v)^{n-i} q_v^{\frac{n-1}{2n}},\;
t_v= (-1)^{n-1} q_v^{\frac{n^2-1}{n}},\; a_v =   q_v^{\frac{n+1-2n^2}{4}}.
%
\end{align*}


\def\fS{\mathfrak{S}_n}

We will  use $\mathfrak{S}_n$ to denote the permutation group on the set $\{1,2,\dots,n\}$.

\def\M {M,\cN}


The stated $SL_n$-skein module of $\MN$, denoted as $S_n(\M,v)$, is
the quotient module of the $R$-module freely generated by the set 
 of all isotopy classes of stated 
$n$-webs in $\MN$ subject to  relations \eqref{w.cross}-\eqref{wzh.eight}.

\beq\label{w.cross}
q_v^{\frac{1}{n}} 
\raisebox{-.20in}{

\begin{tikzpicture}
\tikzset{->-/.style=

{decoration={markings,mark=at position #1 with

{\arrow{latex}}},postaction={decorate}}}
\filldraw[draw=white,fill=gray!20] (-0,-0.2) rectangle (1, 1.2);
\draw [line width =1pt,decoration={markings, mark=at position 0.5 with {\arrow{>}}},postaction={decorate}](0.6,0.6)--(1,1);
\draw [line width =1pt,decoration={markings, mark=at position 0.5 with {\arrow{>}}},postaction={decorate}](0.6,0.4)--(1,0);
\draw[line width =1pt] (0,0)--(0.4,0.4);
\draw[line width =1pt] (0,1)--(0.4,0.6);
\draw[line width =1pt] (0.4,0.6)--(0.6,0.4);
\end{tikzpicture}
}
- q_v^{-\frac {1}{n}}
\raisebox{-.20in}{
\begin{tikzpicture}
\tikzset{->-/.style=

{decoration={markings,mark=at position #1 with

{\arrow{latex}}},postaction={decorate}}}
\filldraw[draw=white,fill=gray!20] (-0,-0.2) rectangle (1, 1.2);
\draw [line width =1pt,decoration={markings, mark=at position 0.5 with {\arrow{>}}},postaction={decorate}](0.6,0.6)--(1,1);
\draw [line width =1pt,decoration={markings, mark=at position 0.5 with {\arrow{>}}},postaction={decorate}](0.6,0.4)--(1,0);
\draw[line width =1pt] (0,0)--(0.4,0.4);
\draw[line width =1pt] (0,1)--(0.4,0.6);
\draw[line width =1pt] (0.6,0.6)--(0.4,0.4);
\end{tikzpicture}
}
= (q_v-q_v^{-1})
\raisebox{-.20in}{

\begin{tikzpicture}
\tikzset{->-/.style=

{decoration={markings,mark=at position #1 with

{\arrow{latex}}},postaction={decorate}}}
\filldraw[draw=white,fill=gray!20] (-0,-0.2) rectangle (1, 1.2);
\draw [line width =1pt,decoration={markings, mark=at position 0.5 with {\arrow{>}}},postaction={decorate}](0,0.8)--(1,0.8);
\draw [line width =1pt,decoration={markings, mark=at position 0.5 with {\arrow{>}}},postaction={decorate}](0,0.2)--(1,0.2);
\end{tikzpicture}
},
\eeq 
\beq\label{w.twist}
\raisebox{-.15in}{
\begin{tikzpicture}
\tikzset{->-/.style=
{decoration={markings,mark=at position #1 with
{\arrow{latex}}},postaction={decorate}}}
\filldraw[draw=white,fill=gray!20] (-1,-0.35) rectangle (0.6, 0.65);
\draw [line width =1pt,decoration={markings, mark=at position 0.5 with {\arrow{>}}},postaction={decorate}](-1,0)--(-0.25,0);
\draw [color = black, line width =1pt](0,0)--(0.6,0);
\draw [color = black, line width =1pt] (0.166 ,0.08) arc (-37:270:0.2);
\end{tikzpicture}}
= t_v
\raisebox{-.15in}{
\begin{tikzpicture}
\tikzset{->-/.style=
{decoration={markings,mark=at position #1 with
{\arrow{latex}}},postaction={decorate}}}
\filldraw[draw=white,fill=gray!20] (-1,-0.5) rectangle (0.6, 0.5);
\draw [line width =1pt,decoration={markings, mark=at position 0.5 with {\arrow{>}}},postaction={decorate}](-1,0)--(-0.25,0);
\draw [color = black, line width =1pt](-0.25,0)--(0.6,0);
\end{tikzpicture}}
,
\eeq
\beq\label{w.unknot}
\raisebox{-.20in}{
\begin{tikzpicture}
\tikzset{->-/.style=
{decoration={markings,mark=at position #1 with
{\arrow{latex}}},postaction={decorate}}}
\filldraw[draw=white,fill=gray!20] (0,0) rectangle (1,1);
\draw [line width =1pt,decoration={markings, mark=at position 0.5 with {\arrow{>}}},postaction={decorate}](0.45,0.8)--(0.55,0.8);
\draw[line width =1pt] (0.5 ,0.5) circle (0.3);
\end{tikzpicture}}
= (-1)^{n-1} [n]\ 
\raisebox{-.20in}{
\begin{tikzpicture}
\tikzset{->-/.style=
{decoration={markings,mark=at position #1 with
{\arrow{latex}}},postaction={decorate}}}
\filldraw[draw=white,fill=gray!20] (0,0) rectangle (1,1);
\end{tikzpicture}}
,\ \text{where}\ [n]=\frac{q_v^n-q_v^{-n}}{q_v-q_v^{-1}},
\eeq
\beq\label{wzh.four}
\raisebox{-.30in}{
\begin{tikzpicture}
\tikzset{->-/.style=
{decoration={markings,mark=at position #1 with
{\arrow{latex}}},postaction={decorate}}}
\filldraw[draw=white,fill=gray!20] (-1,-0.7) rectangle (1.2,1.3);
\draw [line width =1pt,decoration={markings, mark=at position 0.5 with {\arrow{>}}},postaction={decorate}](-1,1)--(0,0);
\draw [line width =1pt,decoration={markings, mark=at position 0.5 with {\arrow{>}}},postaction={decorate}](-1,0)--(0,0);
\draw [line width =1pt,decoration={markings, mark=at position 0.5 with {\arrow{>}}},postaction={decorate}](-1,-0.4)--(0,0);
\draw [line width =1pt,decoration={markings, mark=at position 0.5 with {\arrow{<}}},postaction={decorate}](1.2,1)  --(0.2,0);
\draw [line width =1pt,decoration={markings, mark=at position 0.5 with {\arrow{<}}},postaction={decorate}](1.2,0)  --(0.2,0);
\draw [line width =1pt,decoration={markings, mark=at position 0.5 with {\arrow{<}}},postaction={decorate}](1.2,-0.4)--(0.2,0);
\node  at(-0.8,0.5) {$\vdots$};
\node  at(1,0.5) {$\vdots$};
\end{tikzpicture}}=(-q_v)^{\frac{n(n-1)}{2}}\cdot \sum_{\sigma\in \mathfrak{S}_n}
(-q_v^{\frac{1-n}n})^{\ell(\sigma)} \raisebox{-.30in}{
\begin{tikzpicture}
\tikzset{->-/.style=
{decoration={markings,mark=at position #1 with
{\arrow{latex}}},postaction={decorate}}}
\filldraw[draw=white,fill=gray!20] (-1,-0.7) rectangle (1.2,1.3);
\draw [line width =1pt,decoration={markings, mark=at position 0.5 with {\arrow{>}}},postaction={decorate}](-1,1)--(0,0);
\draw [line width =1pt,decoration={markings, mark=at position 0.5 with {\arrow{>}}},postaction={decorate}](-1,0)--(0,0);
\draw [line width =1pt,decoration={markings, mark=at position 0.5 with {\arrow{>}}},postaction={decorate}](-1,-0.4)--(0,0);
\draw [line width =1pt,decoration={markings, mark=at position 0.5 with {\arrow{<}}},postaction={decorate}](1.2,1)  --(0.2,0);
\draw [line width =1pt,decoration={markings, mark=at position 0.5 with {\arrow{<}}},postaction={decorate}](1.2,0)  --(0.2,0);
\draw [line width =1pt,decoration={markings, mark=at position 0.5 with {\arrow{<}}},postaction={decorate}](1.2,-0.4)--(0.2,0);
\node  at(-0.8,0.5) {$\vdots$};
\node  at(1,0.5) {$\vdots$};
\filldraw[draw=black,fill=gray!20,line width =1pt]  (0.1,0.3) ellipse (0.4 and 0.7);
\node  at(0.1,0.3){$\sigma_{+}$};
\end{tikzpicture}},
\eeq
where the ellipse enclosing $\sigma_+$  is the minimum crossing positive braid representing a permutation $\sigma\in \mathfrak{S}_n$ and $\ell(\sigma)=\mid\{(i,j)\mid 1\leq i<j\leq n, \sigma(i)>\sigma(j)\}|$ is the length of $\sigma\in \mathfrak{S}_n$.

\beq\label{wzh.five}
   \raisebox{-.30in}{
\begin{tikzpicture}
\tikzset{->-/.style=
{decoration={markings,mark=at position #1 with
{\arrow{latex}}},postaction={decorate}}}
\filldraw[draw=white,fill=gray!20] (-1,-0.7) rectangle (0.2,1.3);
\draw [line width =1pt](-1,1)--(0,0);
\draw [line width =1pt](-1,0)--(0,0);
\draw [line width =1pt](-1,-0.4)--(0,0);
\draw [line width =1.5pt](0.2,1.3)--(0.2,-0.7);
\node  at(-0.8,0.5) {$\vdots$};
\filldraw[fill=white,line width =0.8pt] (-0.5 ,0.5) circle (0.07);
\filldraw[fill=white,line width =0.8pt] (-0.5 ,0) circle (0.07);
\filldraw[fill=white,line width =0.8pt] (-0.5 ,-0.2) circle (0.07);
\end{tikzpicture}}
   = 
   a_v \sum_{\sigma \in \mathfrak{S}_n} (-q_v)^{\ell(\sigma)}\,  \raisebox{-.30in}{
\begin{tikzpicture}
\tikzset{->-/.style=
{decoration={markings,mark=at position #1 with
{\arrow{latex}}},postaction={decorate}}}
\filldraw[draw=white,fill=gray!20] (-1,-0.7) rectangle (0.2,1.3);
\draw [line width =1pt](-1,1)--(0.2,1);
\draw [line width =1pt](-1,0)--(0.2,0);
\draw [line width =1pt](-1,-0.4)--(0.2,-0.4);
\draw [line width =1.5pt,decoration={markings, mark=at position 1 with {\arrow{>}}},postaction={decorate}](0.2,1.3)--(0.2,-0.7);
\node  at(-0.8,0.5) {$\vdots$};
\filldraw[fill=white,line width =0.8pt] (-0.5 ,1) circle (0.07);
\filldraw[fill=white,line width =0.8pt] (-0.5 ,0) circle (0.07);
\filldraw[fill=white,line width =0.8pt] (-0.5 ,-0.4) circle (0.07);
\node [right] at(0.2,1) {$\sigma(n)$};
\node [right] at(0.2,0) {$\sigma(2)$};
\node [right] at(0.2,-0.4){$\sigma(1)$};
\end{tikzpicture}},
\eeq
\beq \label{wzh.six}
\raisebox{-.20in}{
\begin{tikzpicture}
\tikzset{->-/.style=
{decoration={markings,mark=at position #1 with
{\arrow{latex}}},postaction={decorate}}}
\filldraw[draw=white,fill=gray!20] (-0.7,-0.7) rectangle (0,0.7);
\draw [line width =1.5pt,decoration={markings, mark=at position 1 with {\arrow{>}}},postaction={decorate}](0,0.7)--(0,-0.7);
\draw [color = black, line width =1pt] (0 ,0.3) arc (90:270:0.5 and 0.3);
\node [right]  at(0,0.3) {$i$};
\node [right] at(0,-0.3){$j$};
\filldraw[fill=white,line width =0.8pt] (-0.5 ,0) circle (0.07);
\end{tikzpicture}}   = \delta_{\bar j,i }\,  c_{i,v}\ \raisebox{-.20in}{
\begin{tikzpicture}
\tikzset{->-/.style=
{decoration={markings,mark=at position #1 with
{\arrow{latex}}},postaction={decorate}}}
\filldraw[draw=white,fill=gray!20] (-0.7,-0.7) rectangle (0,0.7);
\draw [line width =1.5pt](0,0.7)--(0,-0.7);
\end{tikzpicture}},
\eeq
\beq \label{wzh.seven}
\raisebox{-.20in}{
\begin{tikzpicture}
\tikzset{->-/.style=
{decoration={markings,mark=at position #1 with
{\arrow{latex}}},postaction={decorate}}}
\filldraw[draw=white,fill=gray!20] (-0.7,-0.7) rectangle (0,0.7);
\draw [line width =1.5pt](0,0.7)--(0,-0.7);
\draw [color = black, line width =1pt] (-0.7 ,-0.3) arc (-90:90:0.5 and 0.3);
\filldraw[fill=white,line width =0.8pt] (-0.55 ,0.26) circle (0.07);
\end{tikzpicture}}
= \sum_{i=1}^n  (c_{\bar i,v})^{-1}\, \raisebox{-.20in}{
\begin{tikzpicture}
\tikzset{->-/.style=
{decoration={markings,mark=at position #1 with
{\arrow{latex}}},postaction={decorate}}}
\filldraw[draw=white,fill=gray!20] (-0.7,-0.7) rectangle (0,0.7);
\draw [line width =1.5pt,decoration={markings, mark=at position 1 with {\arrow{>}}},postaction={decorate}](0,0.7)--(0,-0.7);
\draw [line width =1pt](-0.7,0.3)--(0,0.3);
\draw [line width =1pt](-0.7,-0.3)--(0,-0.3);
\filldraw[fill=white,line width =0.8pt] (-0.3 ,0.3) circle (0.07);
\filldraw[fill=black,line width =0.8pt] (-0.3 ,-0.3) circle (0.07);
\node [right]  at(0,0.3) {$i$};
\node [right]  at(0,-0.3) {$\bar{i}$};
\end{tikzpicture}},
\eeq
\beq\label{wzh.eight}
\raisebox{-.20in}{

\begin{tikzpicture}
\tikzset{->-/.style=

{decoration={markings,mark=at position #1 with

{\arrow{latex}}},postaction={decorate}}}
\filldraw[draw=white,fill=gray!20] (-0,-0.2) rectangle (1, 1.2);
\draw [line width =1.5pt,decoration={markings, mark=at position 1 with {\arrow{>}}},postaction={decorate}](1,1.2)--(1,-0.2);
\draw [line width =1pt](0.6,0.6)--(1,1);
\draw [line width =1pt](0.6,0.4)--(1,0);
\draw[line width =1pt] (0,0)--(0.4,0.4);
\draw[line width =1pt] (0,1)--(0.4,0.6);
\draw[line width =1pt] (0.4,0.6)--(0.6,0.4);
\filldraw[fill=white,line width =0.8pt] (0.2 ,0.2) circle (0.07);
\filldraw[fill=white,line width =0.8pt] (0.2 ,0.8) circle (0.07);
\node [right]  at(1,1) {$i$};
\node [right]  at(1,0) {$j$};
\end{tikzpicture}
} =q_v^{-\frac{1}{n}}\left(\delta_{{j<i} }(q_v-q_v^{-1})\raisebox{-.20in}{

\begin{tikzpicture}
\tikzset{->-/.style=

{decoration={markings,mark=at position #1 with

{\arrow{latex}}},postaction={decorate}}}
\filldraw[draw=white,fill=gray!20] (-0,-0.2) rectangle (1, 1.2);
\draw [line width =1.5pt,decoration={markings, mark=at position 1 with {\arrow{>}}},postaction={decorate}](1,1.2)--(1,-0.2);
\draw [line width =1pt](0,0.8)--(1,0.8);
\draw [line width =1pt](0,0.2)--(1,0.2);
\filldraw[fill=white,line width =0.8pt] (0.2 ,0.8) circle (0.07);
\filldraw[fill=white,line width =0.8pt] (0.2 ,0.2) circle (0.07);
\node [right]  at(1,0.8) {$i$};
\node [right]  at(1,0.2) {$j$};
\end{tikzpicture}
}+q_v^{\delta_{i,j}}\raisebox{-.20in}{

\begin{tikzpicture}
\tikzset{->-/.style=

{decoration={markings,mark=at position #1 with

{\arrow{latex}}},postaction={decorate}}}
\filldraw[draw=white,fill=gray!20] (-0,-0.2) rectangle (1, 1.2);
\draw [line width =1.5pt,decoration={markings, mark=at position 1 with {\arrow{>}}},postaction={decorate}](1,1.2)--(1,-0.2);
\draw [line width =1pt](0,0.8)--(1,0.8);
\draw [line width =1pt](0,0.2)--(1,0.2);
\filldraw[fill=white,line width =0.8pt] (0.2 ,0.8) circle (0.07);
\filldraw[fill=white,line width =0.8pt] (0.2 ,0.2) circle (0.07);
\node [right]  at(1,0.8) {$j$};
\node [right]  at(1,0.2) {$i$};
\end{tikzpicture}
}\right),
\eeq
where   
$\delta_{j<i}= \left \{
 \begin{array}{rr}
     1,                    & j<i\\
     0,                                 & \text{otherwise}
 \end{array}
 \right.,
\delta_{i,j}= \left \{
 \begin{array}{rr}
     1,                    & i=j\\
     0,                                 & \text{otherwise}
 \end{array}
 \right.$. Each shaded rectangle in the above relations is the projection of a small open cube embedded in $M$. The lines contained in the shaded rectangle represent parts of stated $n$-webs with framing  pointing to  readers. The thick line in the edge of shaded rectangle is a part of the marking.   For detailed explanation for the above relations, please refer to \cite{le2021stated}.


\subsection{Functoriality}
In this subsection, we review the functoriality of the stated $SL_n$-skein modules as introduced in \cite{le2021stated}.

For any two marked  3-manifolds $\MN,(M',\cN')$, if an orientation preserving  embedding $f:M\rightarrow M'$ maps 
$\cN$ to $\cN'$ and preserves  orientations between $\cN$ and $\cN$, we call $f$  an embedding from $\MN$ to $(M',\cN')$. Clearly $f$ induces an $R$-linear map $f_{*}:S_n(\M,v)\rightarrow S_n(M',\cN',v)$ \cite{le2021stated}.
If there exist embeddings $f:\MN\rightarrow (M',\cN')$ and  $g:(M',\cN')\rightarrow \MN$ such that $g,f$
are inverse to each other, we say that $\MN$ and $(M',\cN')$
are isomorphic to each other.

\subsection{The splitting map}\label{splitting}
In this subsection, we review the splitting map of the stated $SL_n$-skein modules as introduced in \cite{le2021stated}.

Let $\MN$ be a marked 3-manifold, and let $D$ be a properly embedded closed disk in $M$, called the {\bf splitting disk}, such that $D$ does not intersect the closure of $\cN$. Removing a regular open neighborhood of $D$ from $M$ results in a new 3-manifold $M'$, where $\partial M'$ contains two copies, $D_1$ and $D_2$, of $D$. The original manifold $M$ can be recovered from $M'$ by gluing $D_1$ and $D_2$. Let $\text{pr}$ denote the natural projection from $M'$ to $M$.

Let $e \subset D$ be an oriented open interval. Suppose $\text{pr}^{-1}(e) = e_1 \cup e_2$ with $e_1 \subset D_1$ and $e_2 \subset D_2$. Cutting $(M, \cN)$ along $(D, e)$ yields a new marked 3-manifold $(M', \cN')$, where $\cN' = \cN \cup e_1 \cup e_2$. We will  denote   $(M{'},\cN{'})$ as  Cut$_{(D,e)}(M, \cN )$.
It is easy to see that $\text{Cut}_{(D, e)}(M, \cN)$ is well-defined up to isomorphism. Moreover, if $e'$ is another oriented open interval in $D$, then $\text{Cut}_{(D, e)}(M, \cN)$ is isomorphic to $\text{Cut}_{(D, e')}(M, \cN)$.  

There exists an $R$-linear map, known as the {\bf splitting map},  
$$
\Theta_{(D, e)}: S_n(M, \cN, v) \rightarrow S_n(M', \cN', v),
$$  
as introduced in \cite{le2021stated}.

\subsection{Punctured bordered surfaces and stated $SL_n$-skein algebras}\label{subb2.4}


A {\bf punctured bordered surface} $\Sigma$ is $\overline{\Sigma}-U$ where $\overline{\Sigma}$ is a compact oriented surface and $U$ is a finite set of $\overline{\Sigma}$ such that every component of $\partial \overline{\Sigma}$ intersects $U$.  
 For simplicity, we will call a punctured bordered surface a {\bf pb surface}.

The points in $U$ are called {\bf ideal points}. An embedded smooth curve in the interior of $\Sigma$ connecting  two points in $U$ (these two points could be the same point) is called an {\bf ideal arc}.  


The stated $SL_n$-skein algebra, denoted as $S_n(\Sigma,v)$, of a pb surface $\Sigma $ is defined as following: For every  component $c$ of $\partial \Sigma$, we choose a point $x_c$. Let $M = \Sigma\times[-1,1]$ and $\cN=\cup_{c}\, x_c \times (-1,1)$  where $c$ goes over all components of $\partial \Sigma$. The orientation of $\cN$ is given by the positive orientation of $(-1,1).$ Then we define $S_n(\Sigma,v) $ to be $S_n(M,\cN,v)$. We will call $(M,\cN)$  the thickening of the pb surface $\Sigma$.
Obviously $S_n(\Sigma,v) $ admits an algebra structure. For any two stated $n$-webs $\alpha_1$ and $\alpha_2$ in  the thickening of $\Sigma$, we define $\alpha_1\alpha_2\in S_n({\Sigma},v)$ to be the result of stacking $\alpha_1$ above $\alpha_2$.  


Any stated $n$-web $\alpha$ in the thickening of $\Sigma$ can be represented by a {\bf stated $n$-web diagram} in $\Sigma$, which is obtained by projecting $\alpha$ onto $\Sigma$. Before performing the projection, we first apply a height-preserving isotopy to $\alpha$ so that its singular points look like the pictures in Figure \ref{fg1} and ensure that the endpoints of $\alpha$ are distinct.

The orientation of $\partial \Sigma$ induced by the orientation of $\Sigma$ is call the {\bf positive orientation}
 of $\partial \Sigma$. The orientation of $\partial \Sigma$ opposite to the positive  orientation of $\partial \Sigma$ is called the {\bf negative orientation} of $\partial \Sigma$. A stated $n$-web diagram, where the heights of its endpoints are given by the positive (resp. negative) orientation of $\Sigma$, is called the {\bf positively ordered stated $n$-web diagram} (resp. {\bf negatively ordered stated $n$-web diagram}).



%
%
%
%
%

For a stated $n$-web (resp. a stated $n$-web diagram) $\alpha$, we define the following quantities:  
\begin{itemize}
    \item $e(\alpha)$: the number of endpoints of $\alpha$.

    \item $t(\alpha)$: the number of endpoints of $\alpha$ that point towards $\mathcal{N}$ (resp. towards the boundary of the pb surface).

    \item $p(\alpha)$: the number of sinks and sources of $\alpha$.
\end{itemize}

\subsection{The splitting map for pb surfaces}\label{sub27}

Let $c$ be ideal arc of a pb surface $\Sigma$. After cutting $\Sigma$ along $c$, we get a new pb surface $\text{Cut}_c\,{\Sigma}$, which has two copies $c_1,c_2$ for $c$ such that 
${\Sigma}= \text{Cut}_c\,{\Sigma}/(c_1=c_2)$. We use $\pr$ to denote the projection from $\text{Cut}_c\,{\Sigma}$ to $\Sigma$.  Suppose $\alpha$ is a stated $n$-web diagram in $\Sigma$, which is transverse to $c$.
Let $s$ be a map from $c\cap\alpha$ to $\{1,2,\cdots,n\}$, and $h$ is a linear order on $c\cap\alpha$. Then there is a lift stated $n$-web diagram $\alpha(h,s)$  in $\text{Cut}_c\,{\Sigma}$. The heights of the newly created endpoints of $\alpha(h,s)$ are induced by $h$ (via $\pr$), and the states of the newly created endpoints of $\alpha(h,s)$ are induced by $s$ (via $\pr$).
Then the splitting map is defined by 
$$\Theta_c(\alpha) =\sum_{s: \alpha \cap c \to \{1,\dots, n\}} \alpha(h, s),$$
furthermore $\Theta_c$ is an algebra homomorphism \cite{le2021stated}. 

\subsection{The $O_q(SL_n)$}\label{bigon}

We refer to  \cite{KS,le2021stated,leY} for the definition of $O_q(SL_n)$.

We have the following coefficients
\beq
\cR^{ij}_{lk} = q^{-\frac 1n} \left(    q^{ \delta_{i,j}} \delta_{j,k} \delta_{i,l} + (q-q^{-1})
\delta_{j<i} \delta_{j,l} \delta_{i,k}\right),
\label{R}
\eeq
where $\delta_{j<i}=1$ if $j<i$ and $\delta_{j<i}=0$ otherwise.

Let $O_q(M(n))$ be the associative algebra generated by  $u_{i,j}$,   $i,j\in\{1,\cdots,n\},$
subject to the relations 
\beq
(\buu \ot \buu) \cR = \cR (\buu \ot \buu),  
\eeq
where $\cR$ is the $n^2\times n^2$ matrix given by equation \eqref{R}, and $\buu \ot \buu$ is the $n^2\times n^2$ matrix with entries $(\buu \ot \buu)^{ik}_{jl} = u_{i,j} u_{k,l}$ for $i,j,k,l\in \{1,\cdots,n\}$. 
Define  the element 
$$ {\det}_q(\buu)\triangleq \sum_{\sigma\in \fS} (-q)^{\ell(\sigma)}u_{1,\sigma(1)}\cdots u_{n,\sigma(n)} = \sum_{\sigma\in \fS} (-q)^{\ell(\sigma)}u_{\sigma(1),1}\cdots u_{\sigma(n),n}.$$

Define $O_q(SL_n)$ to be  $O_q(M(n))/(\det_q \buu-1).$ Then 
$O_q(SL_n)$ is a Hopf algebra with the Hopf algebra structure given by
\begin{align*}
	\Delta(u_{i,j}) & = \sum_{k=1}^n u_{i,k} \ot u_{k,j}, \quad  \epsilon(u_{i,j})= \delta_{i,j}.\label{eq.Deltave}\\
	S({u}_{i,j} )&= (\buu^!)_{i,j} = (-q)^{i-j} {\det}_q(\buu^{j,i}).
\end{align*}
Here $\buu^{j,i}$ is the result of removing the $j$-th row and $i$-th column from $\buu$.

We have that $O_{q}(SL_n)$ is isomorphic to the stated skein algebra of the bigon as Hopf algebras \cite{le2021stated}.




\section{Stated $SL_n$-skein algebras skewed by  roots of unity}

For a pb surface $\Sigma$,
two negatively ordered stated $n$-web diagrams represent the isotopic stated $n$-webs in $\Sigma\times [-1,1]$ if and only if they are related by a sequence of the ambient isotopies and the five moves in Figure \ref{fg}. See Figure 2 in \cite{frohman20223} for $n=3$.

\begin{figure}[h]  
	\centering\includegraphics[width=14cm]{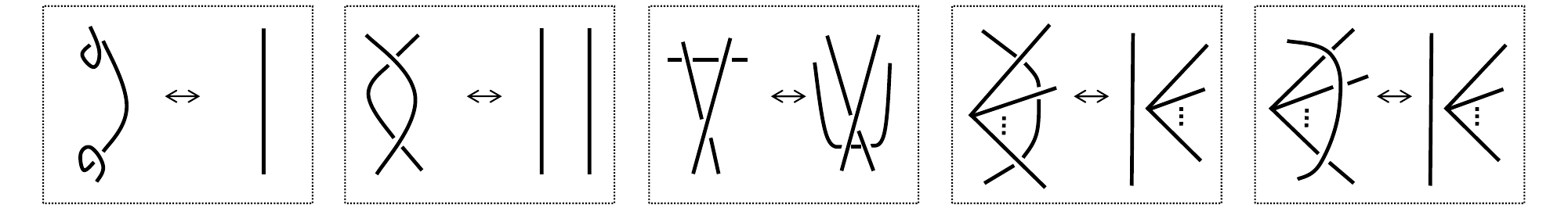} 
	\caption{Five moves for the (stated) $n$-web diagrams. The orientations for (stated) $n$-web diagrams are arbitrary.} 
	\label{fg}   
\end{figure}

For a pb surface $\Sigma$, we can regard $S_n(\Sigma,v)$ as the quotient of the $R$-module freely generated by the set of ambient isotopy classes of negatively ordered stated $n$-web diagrams, subject to the five moves in Figure \ref{fg} and relations \eqref{w.cross}-\eqref{wzh.eight}. The multiplication is then defined by stacking negatively ordered stated $n$-web diagrams in such a way that the result remains a negatively ordered stated $n$-web diagram.

For a negatively ordered stated $n$-web diagram $\alpha$, every crossing point $p$ of $\alpha$ determines a ``$+$" or ``$-$" sign, as illustrated in Figure \ref{fg1}, which we denote by $w(p)$.  
We then define $w(\alpha) = \sum_{p} w(p)$, where the sum is taken over all crossing points of $\alpha$, with the plus sign regarded as $1$ and the minus sign as $-1$.  
If $\alpha$ has no crossings, we define $w(\alpha) = 0$.

\begin{figure}[h]  
	\centering\includegraphics[width=5cm]{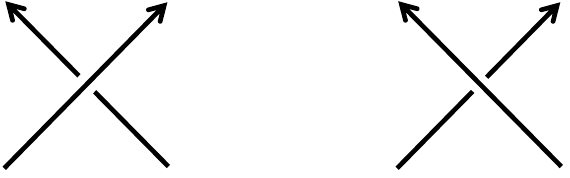} 
	\caption{The sign determined by the left (resp. the right) picture is ``$+$" (resp. ``$-$").}       
	\label{fg1}   
\end{figure}

For any two integers $i,j$ with $j>0$, suppose $i = kj+r$, where $k,r\in \mathbb{Z}$ and 
$0\leq r\leq j-1$. We define $i(\text{mod}\;j)$ to be $r$.

Let $\alpha,\beta$ be two negatively ordered stated $n$-web diagrams. Suppose $\gamma$ is a negative ordered stated $n$-web diagram representing $\alpha\beta$. 
Then, define  
$$
w(\alpha, \beta) = \left(\sum_{p} w(p)\right) (\text{mod}\;n),
$$
where the sum is taken over the subset of crossing points in $\gamma$ that involve strands from both $\alpha$ and $\beta$. 
From the moves in Figure \ref{fg}, we conclude that \( w(\alpha, \beta) \) is well-defined and satisfies \( w(\alpha', \beta') = w(\alpha, \beta) \) whenever \(\alpha\) (resp. \(\beta\)) and \(\alpha'\) (resp. \(\beta'\)) represent isotopic stated \( n \)-webs.  
Note that \( w(\alpha, \beta) \) does not necessarily satisfy  
\[
w(\alpha, \beta) = (n - w(\beta, \alpha) )(\text{mod}\;n).
\]  
However, this equality holds when both \(\alpha\) and \(\beta\) have no endpoints.  

For any integer \( 0 \leq k \leq n - 1 \), we define \( S_n(\Sigma, v)_k \) as  
\begin{gather}
    R\text{-linear span} \{ \text{negatively ordered stated } n\text{-web diagrams } \alpha \mid  \nonumber \\
    w(\alpha, \alpha') \text{ and } w(\alpha', \alpha) \text{ are multiples of } k,  \label{condition} \\
    \text{for all negatively ordered stated } n\text{-web diagrams } \alpha' \}. \nonumber
\end{gather}  
It is straightforward to verify that \( S_n(\Sigma, v)_k \) forms a subalgebra of \( S_n(\Sigma, v) \) whenever \( k \mid n \).

\subsection{On $SL_n$-skein algebras}

Let \(\Sigma\) be a surface with empty boundary. Then, the definition of \( S_n(\Sigma, v) \) depends only on \( v^2 \), meaning that \( S_n(\Sigma, v) \) and \( S_n(\Sigma, -v) \) define the same (stated) \( SL_n \)-skein algebra. Consequently, when \( \partial \Sigma = \emptyset \), we use the notation \( S_n(\Sigma|v^2) \) to represent \( S_n(\Sigma, v) \).  
Since \( \partial \Sigma = \emptyset \), every (stated) \( n \)-web diagram in \( \Sigma \) is necessarily a negatively ordered stated \( n \)-web diagram.

\begin{theorem}\label{skein}
	Let $\Sigma$ be a surface with $\partial \Sigma =\emptyset$, let $m,n$ be two positive integers with $m\mid n$, and let $u,\epsilon\in R$ be two invertible elements in $R$  such that $u^2 = \epsilon v^2$, $\epsilon^m = 1$.  Then there exists an $R$-linear isomorphism $\varphi_{\epsilon}:S_n(\Sigma |v^2)\rightarrow S_n(\Sigma|u^2)$, defined by $\varphi_{\epsilon}(\alpha) = \epsilon^{w(\alpha)}\alpha$ for any $n$-web diagram $\alpha$.
\end{theorem}
\begin{proof}
	
	Suppose that $\alpha_1$ and $\alpha_2$ are $n$-web diagrams representing the isotopic $n$-webs in $\Sigma\times[-1,1]$. The five moves in Figure \ref{fg} imply $w(\alpha_1)(\text{mod}\;n) = w(\alpha_2)(\text{mod}\;n)$. Then we have $\epsilon^{w(\alpha_1)} = \epsilon^{w(\alpha_2)}$ because $\epsilon^n = 1$. Thus $\varphi_{\epsilon}$ is well-defined on the set of isotopy classes of $n$-webs in $\Sigma\times [-1,1]$.
	
	It suffices to show that $\varphi_{\epsilon}$ preserves relations \eqref{w.cross}-\eqref{wzh.four}. Note that $q_u^{\frac{1}{n}} = \epsilon q_v^{\frac{1}{n}}$,  $q_u= q_v$, and $t_u = (-1)^{n-1}q_u^{n-\frac{1}{n}} = \epsilon^{-1}t_v$.

Relation \eqref{w.cross}: We use $\alpha_1$ to denote the $n$-web diagram on the right-hand side of relation \eqref{w.cross}. Then we have 
	\begin{align*}
		&\varphi_{\epsilon}(
	q_v^{\frac{1}{n}} 
	\raisebox{-.20in}{		
		\begin{tikzpicture}
			\tikzset{->-/.style=				
				{decoration={markings,mark=at position #1 with						
						{\arrow{latex}}},postaction={decorate}}}
			\filldraw[draw=white,fill=gray!20] (-0,-0.2) rectangle (1, 1.2);
			\draw [line width =1pt,decoration={markings, mark=at position 0.5 with {\arrow{>}}},postaction={decorate}](0.6,0.6)--(1,1);
			\draw [line width =1pt,decoration={markings, mark=at position 0.5 with {\arrow{>}}},postaction={decorate}](0.6,0.4)--(1,0);
			\draw[line width =1pt] (0,0)--(0.4,0.4);
			\draw[line width =1pt] (0,1)--(0.4,0.6);
			\draw[line width =1pt] (0.4,0.6)--(0.6,0.4);
		\end{tikzpicture}	}
	- q_v^{-\frac {1}{n}}
	\raisebox{-.20in}{
		\begin{tikzpicture}
			\tikzset{->-/.style=				
				{decoration={markings,mark=at position #1 with						
						{\arrow{latex}}},postaction={decorate}}}
			\filldraw[draw=white,fill=gray!20] (-0,-0.2) rectangle (1, 1.2);
			\draw [line width =1pt,decoration={markings, mark=at position 0.5 with {\arrow{>}}},postaction={decorate}](0.6,0.6)--(1,1);
			\draw [line width =1pt,decoration={markings, mark=at position 0.5 with {\arrow{>}}},postaction={decorate}](0.6,0.4)--(1,0);
			\draw[line width =1pt] (0,0)--(0.4,0.4);
			\draw[line width =1pt] (0,1)--(0.4,0.6);
			\draw[line width =1pt] (0.6,0.6)--(0.4,0.4);
		\end{tikzpicture}
	})=\epsilon^{w(\alpha_1)} \epsilon
	q_v^{\frac{1}{n}} 
	\raisebox{-.20in}{		
	\begin{tikzpicture}
		\tikzset{->-/.style=				
			{decoration={markings,mark=at position #1 with						
					{\arrow{latex}}},postaction={decorate}}}
		\filldraw[draw=white,fill=gray!20] (-0,-0.2) rectangle (1, 1.2);
		\draw [line width =1pt,decoration={markings, mark=at position 0.5 with {\arrow{>}}},postaction={decorate}](0.6,0.6)--(1,1);
		\draw [line width =1pt,decoration={markings, mark=at position 0.5 with {\arrow{>}}},postaction={decorate}](0.6,0.4)--(1,0);
		\draw[line width =1pt] (0,0)--(0.4,0.4);
		\draw[line width =1pt] (0,1)--(0.4,0.6);
		\draw[line width =1pt] (0.4,0.6)--(0.6,0.4);
		\end{tikzpicture}	}
		- \epsilon^{w(\alpha_1)} \epsilon^{-1} q_v^{-\frac {1}{n}}
		\raisebox{-.20in}{
	\begin{tikzpicture}
		\tikzset{->-/.style=				
			{decoration={markings,mark=at position #1 with						
					{\arrow{latex}}},postaction={decorate}}}
		\filldraw[draw=white,fill=gray!20] (-0,-0.2) rectangle (1, 1.2);
		\draw [line width =1pt,decoration={markings, mark=at position 0.5 with {\arrow{>}}},postaction={decorate}](0.6,0.6)--(1,1);
		\draw [line width =1pt,decoration={markings, mark=at position 0.5 with {\arrow{>}}},postaction={decorate}](0.6,0.4)--(1,0);
		\draw[line width =1pt] (0,0)--(0.4,0.4);
		\draw[line width =1pt] (0,1)--(0.4,0.6);
		\draw[line width =1pt] (0.6,0.6)--(0.4,0.4);
	\end{tikzpicture}
	}\\
	= &(q_u-q_u^{-1})\epsilon^{w(\alpha_1)}
	\raisebox{-.20in}{		
		\begin{tikzpicture}
			\tikzset{->-/.style=				
				{decoration={markings,mark=at position #1 with						
						{\arrow{latex}}},postaction={decorate}}}
			\filldraw[draw=white,fill=gray!20] (-0,-0.2) rectangle (1, 1.2);
			\draw [line width =1pt,decoration={markings, mark=at position 0.5 with {\arrow{>}}},postaction={decorate}](0,0.8)--(1,0.8);
			\draw [line width =1pt,decoration={markings, mark=at position 0.5 with {\arrow{>}}},postaction={decorate}](0,0.2)--(1,0.2);
		\end{tikzpicture}
	}= \varphi_{\epsilon}((q_v-q_v^{-1})
	\raisebox{-.20in}{		
	\begin{tikzpicture}
		\tikzset{->-/.style=				
			{decoration={markings,mark=at position #1 with						
					{\arrow{latex}}},postaction={decorate}}}
		\filldraw[draw=white,fill=gray!20] (-0,-0.2) rectangle (1, 1.2);
		\draw [line width =1pt,decoration={markings, mark=at position 0.5 with {\arrow{>}}},postaction={decorate}](0,0.8)--(1,0.8);
		\draw [line width =1pt,decoration={markings, mark=at position 0.5 with {\arrow{>}}},postaction={decorate}](0,0.2)--(1,0.2);
	\end{tikzpicture}
	}).
\end{align*}

Relation \eqref{w.twist}: We use $\alpha_2$ to denote the $n$-web diagram on the right-hand side of relation \eqref{w.twist}. Then we have 
	\begin{align*}
		\varphi_{\epsilon}(
	\raisebox{-.15in}{
		\begin{tikzpicture}
			\tikzset{->-/.style=
				{decoration={markings,mark=at position #1 with
						{\arrow{latex}}},postaction={decorate}}}
			\filldraw[draw=white,fill=gray!20] (-1,-0.35) rectangle (0.6, 0.65);
			\draw [line width =1pt,decoration={markings, mark=at position 0.5 with {\arrow{>}}},postaction={decorate}](-1,0)--(-0.25,0);
			\draw [color = black, line width =1pt](0,0)--(0.6,0);
			\draw [color = black, line width =1pt] (0.166 ,0.08) arc (-37:270:0.2);
	\end{tikzpicture}})
=\epsilon^{w(\alpha_2)} \epsilon \raisebox{-.15in}{
	\begin{tikzpicture}
		\tikzset{->-/.style=
			{decoration={markings,mark=at position #1 with
					{\arrow{latex}}},postaction={decorate}}}
		\filldraw[draw=white,fill=gray!20] (-1,-0.35) rectangle (0.6, 0.65);
		\draw [line width =1pt,decoration={markings, mark=at position 0.5 with {\arrow{>}}},postaction={decorate}](-1,0)--(-0.25,0);
		\draw [color = black, line width =1pt](0,0)--(0.6,0);
		\draw [color = black, line width =1pt] (0.166 ,0.08) arc (-37:270:0.2);
\end{tikzpicture}}
	= \epsilon^{w(\alpha_2)} \epsilon t_u
	\raisebox{-.15in}{
		\begin{tikzpicture}
			\tikzset{->-/.style=
				{decoration={markings,mark=at position #1 with
						{\arrow{latex}}},postaction={decorate}}}
			\filldraw[draw=white,fill=gray!20] (-1,-0.5) rectangle (0.6, 0.5);
			\draw [line width =1pt,decoration={markings, mark=at position 0.5 with {\arrow{>}}},postaction={decorate}](-1,0)--(-0.25,0);
			\draw [color = black, line width =1pt](-0.25,0)--(0.6,0);
	\end{tikzpicture}}=\varphi_{\epsilon}(t_v\raisebox{-.15in}{
	\begin{tikzpicture}
	\tikzset{->-/.style=
		{decoration={markings,mark=at position #1 with
				{\arrow{latex}}},postaction={decorate}}}
	\filldraw[draw=white,fill=gray!20] (-1,-0.35) rectangle (0.6, 0.65);
	\draw [line width =1pt,decoration={markings, mark=at position 0.5 with {\arrow{>}}},postaction={decorate}](-1,0)--(-0.25,0);
	\draw [color = black, line width =1pt](0,0)--(0.6,0);
	\draw [color = black, line width =1pt] (0.166 ,0.08) arc (-37:270:0.2);
\end{tikzpicture}}).
	\end{align*}

	It is trivial that $\varphi_{\epsilon}$ preserves relation \eqref{w.unknot}.

Relation \eqref{wzh.four}: We use $\alpha_3$ to denote the $n$-web diagram on the left-hand side of relation \eqref{wzh.four}. Then we have 
	\begin{align*}
	&	\varphi_{\epsilon}(
	\raisebox{-.30in}{
		\begin{tikzpicture}
			\tikzset{->-/.style=
				{decoration={markings,mark=at position #1 with
						{\arrow{latex}}},postaction={decorate}}}
			\filldraw[draw=white,fill=gray!20] (-1,-0.7) rectangle (1.2,1.3);
			\draw [line width =1pt,decoration={markings, mark=at position 0.5 with {\arrow{>}}},postaction={decorate}](-1,1)--(0,0);
			\draw [line width =1pt,decoration={markings, mark=at position 0.5 with {\arrow{>}}},postaction={decorate}](-1,0)--(0,0);
			\draw [line width =1pt,decoration={markings, mark=at position 0.5 with {\arrow{>}}},postaction={decorate}](-1,-0.4)--(0,0);
			\draw [line width =1pt,decoration={markings, mark=at position 0.5 with {\arrow{<}}},postaction={decorate}](1.2,1)  --(0.2,0);
			\draw [line width =1pt,decoration={markings, mark=at position 0.5 with {\arrow{<}}},postaction={decorate}](1.2,0)  --(0.2,0);
			\draw [line width =1pt,decoration={markings, mark=at position 0.5 with {\arrow{<}}},postaction={decorate}](1.2,-0.4)--(0.2,0);
			\node  at(-0.8,0.5) {$\vdots$};
			\node  at(1,0.5) {$\vdots$};
	\end{tikzpicture}})
=\epsilon^{w(\alpha_3)}
\raisebox{-.30in}{
	\begin{tikzpicture}
		\tikzset{->-/.style=
			{decoration={markings,mark=at position #1 with
					{\arrow{latex}}},postaction={decorate}}}
		\filldraw[draw=white,fill=gray!20] (-1,-0.7) rectangle (1.2,1.3);
		\draw [line width =1pt,decoration={markings, mark=at position 0.5 with {\arrow{>}}},postaction={decorate}](-1,1)--(0,0);
		\draw [line width =1pt,decoration={markings, mark=at position 0.5 with {\arrow{>}}},postaction={decorate}](-1,0)--(0,0);
		\draw [line width =1pt,decoration={markings, mark=at position 0.5 with {\arrow{>}}},postaction={decorate}](-1,-0.4)--(0,0);
		\draw [line width =1pt,decoration={markings, mark=at position 0.5 with {\arrow{<}}},postaction={decorate}](1.2,1)  --(0.2,0);
		\draw [line width =1pt,decoration={markings, mark=at position 0.5 with {\arrow{<}}},postaction={decorate}](1.2,0)  --(0.2,0);
		\draw [line width =1pt,decoration={markings, mark=at position 0.5 with {\arrow{<}}},postaction={decorate}](1.2,-0.4)--(0.2,0);
		\node  at(-0.8,0.5) {$\vdots$};
		\node  at(1,0.5) {$\vdots$};
\end{tikzpicture}}\\
=&\epsilon^{w(\alpha_3)}(-q_u)^{\frac{n(n-1)}{2}}\cdot \sum_{\sigma\in \fS}
	(-q_u^{\frac{1-n}n})^{\ell(\sigma)} \raisebox{-.30in}{
		\begin{tikzpicture}
			\tikzset{->-/.style=
				{decoration={markings,mark=at position #1 with
						{\arrow{latex}}},postaction={decorate}}}
			\filldraw[draw=white,fill=gray!20] (-1,-0.7) rectangle (1.2,1.3);
			\draw [line width =1pt,decoration={markings, mark=at position 0.5 with {\arrow{>}}},postaction={decorate}](-1,1)--(0,0);
			\draw [line width =1pt,decoration={markings, mark=at position 0.5 with {\arrow{>}}},postaction={decorate}](-1,0)--(0,0);
			\draw [line width =1pt,decoration={markings, mark=at position 0.5 with {\arrow{>}}},postaction={decorate}](-1,-0.4)--(0,0);
			\draw [line width =1pt,decoration={markings, mark=at position 0.5 with {\arrow{<}}},postaction={decorate}](1.2,1)  --(0.2,0);
			\draw [line width =1pt,decoration={markings, mark=at position 0.5 with {\arrow{<}}},postaction={decorate}](1.2,0)  --(0.2,0);
			\draw [line width =1pt,decoration={markings, mark=at position 0.5 with {\arrow{<}}},postaction={decorate}](1.2,-0.4)--(0.2,0);
			\node  at(-0.8,0.5) {$\vdots$};
			\node  at(1,0.5) {$\vdots$};
			\filldraw[draw=black,fill=gray!20,line width =1pt]  (0.1,0.3) ellipse (0.4 and 0.7);
			\node  at(0.1,0.3){$\sigma_{+}$};
	\end{tikzpicture}}\\=&(-q_v)^{\frac{n(n-1)}{2}}\cdot \sum_{\sigma\in \fS}
(-q_v^{\frac{1-n}n})^{\ell(\sigma)} \epsilon^{\ell(\sigma)+w(\alpha_3)} \raisebox{-.30in}{
\begin{tikzpicture}
\tikzset{->-/.style=
	{decoration={markings,mark=at position #1 with
			{\arrow{latex}}},postaction={decorate}}}
\filldraw[draw=white,fill=gray!20] (-1,-0.7) rectangle (1.2,1.3);
\draw [line width =1pt,decoration={markings, mark=at position 0.5 with {\arrow{>}}},postaction={decorate}](-1,1)--(0,0);
\draw [line width =1pt,decoration={markings, mark=at position 0.5 with {\arrow{>}}},postaction={decorate}](-1,0)--(0,0);
\draw [line width =1pt,decoration={markings, mark=at position 0.5 with {\arrow{>}}},postaction={decorate}](-1,-0.4)--(0,0);
\draw [line width =1pt,decoration={markings, mark=at position 0.5 with {\arrow{<}}},postaction={decorate}](1.2,1)  --(0.2,0);
\draw [line width =1pt,decoration={markings, mark=at position 0.5 with {\arrow{<}}},postaction={decorate}](1.2,0)  --(0.2,0);
\draw [line width =1pt,decoration={markings, mark=at position 0.5 with {\arrow{<}}},postaction={decorate}](1.2,-0.4)--(0.2,0);
\node  at(-0.8,0.5) {$\vdots$};
\node  at(1,0.5) {$\vdots$};
\filldraw[draw=black,fill=gray!20,line width =1pt]  (0.1,0.3) ellipse (0.4 and 0.7);
\node  at(0.1,0.3){$\sigma_{+}$};
\end{tikzpicture}}\\=&\varphi_{\epsilon}((-q_v)^{\frac{n(n-1)}{2}}\cdot \sum_{\sigma\in \fS}
(-q_v^{\frac{1-n}n})^{\ell(\sigma)} \raisebox{-.30in}{
\begin{tikzpicture}
\tikzset{->-/.style=
	{decoration={markings,mark=at position #1 with
			{\arrow{latex}}},postaction={decorate}}}
\filldraw[draw=white,fill=gray!20] (-1,-0.7) rectangle (1.2,1.3);
\draw [line width =1pt,decoration={markings, mark=at position 0.5 with {\arrow{>}}},postaction={decorate}](-1,1)--(0,0);
\draw [line width =1pt,decoration={markings, mark=at position 0.5 with {\arrow{>}}},postaction={decorate}](-1,0)--(0,0);
\draw [line width =1pt,decoration={markings, mark=at position 0.5 with {\arrow{>}}},postaction={decorate}](-1,-0.4)--(0,0);
\draw [line width =1pt,decoration={markings, mark=at position 0.5 with {\arrow{<}}},postaction={decorate}](1.2,1)  --(0.2,0);
\draw [line width =1pt,decoration={markings, mark=at position 0.5 with {\arrow{<}}},postaction={decorate}](1.2,0)  --(0.2,0);
\draw [line width =1pt,decoration={markings, mark=at position 0.5 with {\arrow{<}}},postaction={decorate}](1.2,-0.4)--(0.2,0);
\node  at(-0.8,0.5) {$\vdots$};
\node  at(1,0.5) {$\vdots$};
\filldraw[draw=black,fill=gray!20,line width =1pt]  (0.1,0.3) ellipse (0.4 and 0.7);
\node  at(0.1,0.3){$\sigma_{+}$};
\end{tikzpicture}}).
	\end{align*}
	
	Obviously, the inverse of $\varphi_{\epsilon}$ is $\varphi_{\epsilon^{-1}}:S_n(\Sigma|u^2)\rightarrow S_n(\Sigma|v^2)$.
\end{proof}

The map $\varphi_{\epsilon}$ is not an algebra homorphism because $\epsilon^{w(\alpha\beta)}\neq \epsilon^{w(\alpha)}\epsilon^{w(\beta)}$ in general. The following Corollary shows that it restricts to an algebra isomorphism from $S_n(\Sigma |v^2)_m$ to $S_n(\Sigma |u^2)_m$. 

\begin{corollary}\label{skein_cor}
	Let $\Sigma$ be a surface with $\partial \Sigma =\emptyset$, let $m,n$ be two positive integers with $m\mid n$, and let $u,\epsilon\in R$ be two invertible elements in $R$  such that $u^2 = \epsilon v^2$, $\epsilon^m = 1$. Then the linear isomorphism $\varphi_{\epsilon}:S_n(\Sigma|v^2)\rightarrow S_n(\Sigma| u^2)$ restricts to an algebra isomorphism $\varphi_{\epsilon}|_{S_n(\Sigma|v)_m}:S_n(\Sigma|v^2)_m\rightarrow S_n(\Sigma|u^2)_m$.
\end{corollary}
\begin{proof}
	It is trivial that $\varphi_{\epsilon}$ restricts to a linear isomorphism $\varphi_{\epsilon}|_{S_n(\Sigma|v^2)_m}$. Let $\alpha, \alpha'$ be two $n$-web diagrams in $\Sigma$ that satisfy the condition in equation \eqref{condition}. We know that $w(\alpha,\alpha') = (w(\alpha\alpha')-w(\alpha)-w(\alpha'))(\text{mod}\;n)$. Since $m\mid w(\alpha,\alpha')$ and $m\mid n$, we have that $m\mid (w(\alpha\alpha')-w(\alpha)-w(\alpha'))$. Thus $$\varphi_{\epsilon}(\alpha\alpha') = \epsilon^{w(\alpha\alpha')}\alpha\alpha' = \epsilon^{w(\alpha)+w(\alpha')}\alpha\alpha' = \varphi_{\epsilon}(\alpha)\varphi_{\epsilon}(\alpha').$$
\end{proof}

\begin{corollary}
	Suppose that $\Sigma$ is a surface with an empty boundary, and $2\mid n$. Then the linear isomorphism $\varphi_{-1}:S_n(\Sigma|v^2)\rightarrow S_n(\Sigma| -v^2)$ is an algebra isomorphism.
\end{corollary}
\begin{proof}
	Since $\Sigma$ is a surface with an empty boundary, we have that $S_n(\Sigma|v^2)_2 = S_n(\Sigma|v^2)$. Then Corollary \ref{skein_cor} completes the proof.
\end{proof}

\subsection{On stated $SL_n$-skein algebras}
Let $\Sigma$ be a pb surface, let $m,n$ be two positive integers with $m\mid n$, and let $\epsilon\in R$ with $\epsilon^{2m} = 1$. In this subsection, we will construct an $R$-linear isomorphism $\varphi_{\epsilon}:S_n(\Sigma,v)\rightarrow S_n(\Sigma,\epsilon v)$.  We still use the notation $\varphi_{\epsilon}$ for this  $R$-linear isomorphism because it coincides with the  $R$-linear isomorphism built in Theorem \ref{skein} when $\partial \Sigma=\emptyset$.
We present Theorem \ref{skein} and the subsequent theorems separately to aid the reader's understanding and to simplify the proofs of the following results.

Recall that,
 for a 
 stated $n$-web diagram $\alpha$, we use $e(\alpha)$ to denote the number of endpoints of $\alpha$,   use $t(\alpha)$ to denote the number of endpoints  of $\alpha$ that are pointing towards  the boundary of the pb surface, and use $p(\alpha)$ to denote the number of sinks and sources of $\alpha$.

\def\M{M,\mathcal{N}}
\def\MN{(M,\mathcal{N})}

In the following theorem, we construct an $R$-linear isomorphism from $S_n(\Sigma,v)$ to $S_n(\Sigma,\epsilon v)$ when $\epsilon^m = -1$ and $m\mid n$. We will show that this  $R$-linear isomorphism  restricts an algebra isomorphism from $S_n(\Sigma,v)_m$ to $S_n(\Sigma,\epsilon v)_m$.

\begin{theorem}\label{stated1}
	Let $\Sigma$ be a pb surface, let $k,m,n$ be three positive integers with $n=km$, and let $\epsilon\in R$ with $\epsilon^{m} = -1$. Then there exists  an $R$-linear isomorphism $\varphi_{\epsilon}:S_n(\Sigma,v)\rightarrow S_n(\Sigma,\epsilon v)$, defined by
	\begin{equation*}
		\varphi_{\epsilon}(\alpha)=
		\begin{cases}
			\epsilon^{2w(\alpha)+\frac{e(\alpha)}{2}} \alpha& k \text{ is even} \\
			(-\epsilon)^{2w(\alpha)+\frac{e(\alpha)}{2}} \alpha& k \text{ is odd and }m\text{ is even}\\
			(-\epsilon)^{2w(\alpha)+t(\alpha)} \alpha& k \text{ is odd, }m\text{ is odd, and }\frac{k+m}{2} \text{ is odd}\\
			(-1)^{p(\alpha)}(-\epsilon)^{2w(\alpha)+t(\alpha)} \alpha& k \text{ is odd, }m\text{ is odd, and }\frac{k+m}{2} \text{ is even,}
		\end{cases}
	\end{equation*}
	where $\alpha\in S_n(\Sigma,v)$ is a negatively ordered stated $n$-web diagram.  In particular $\varphi_{\epsilon}(\alpha) = \epsilon^{2(w(\alpha))}\alpha$ if $\alpha$ is a negatively ordered stated $n$-web diagram without endpoints.
\end{theorem}
Note that when $n$ is even, the non-negative integer $e(\alpha)$ is always even. 

\begin{proof}
		It is obvious that $\varphi_{\epsilon}$ preserves the five moves in Figure \ref{fg}. The proof for Theorem \ref{skein} implies that $\varphi_{\epsilon}$ preserves relations \eqref{w.cross}-\eqref{wzh.four} because $(\epsilon^2)^m =1$ and $m\mid n$.  Trivially $\varphi_{\epsilon^{-1}}:S_n(\Sigma,\epsilon v)\rightarrow S_n(\Sigma,v)$ is the inverse of $\varphi_{\epsilon}$ if $\varphi_{\epsilon}$ is a well-defined $R$-linear map.
		Then it suffices to show that $\varphi_{\epsilon}$ preserves relations \eqref{wzh.five}-\eqref{wzh.eight}.
	
	Let $u = \epsilon v$. We have $q_u = u^{2n} = \epsilon^{2n}v^{2n} = v^{2n} = q_v$, and $c_{i,u} = (-q_u)^{n-i}(q_u^{\frac{1}{2n}})^{n-1}
	= \epsilon^{n-1 }(-q_v)^{n-i}(q_v^{\frac{1}{2n}})^{n-1}  =\epsilon^{n-1 } c_{i,v}$ for $1\leq i\leq n$.

	 \text{\bf Case 1}: $k$ is even.  Suppose $k = 2\lambda$. We have $\epsilon^n = \epsilon^{mk} = (-1)^k = 1$. Then $c_{i,u}=\epsilon^{n-1 } c_{i,v} = \epsilon^{-1}c_{i,v}$ for $1\leq i\leq n$. Through direct calculations, we obtain $a_u = (-1)^\lambda a_v$.

	 Relation \eqref{wzh.five}: Let $U_1$ be the open square in relation \eqref{wzh.five}, and let $\alpha_1$ and $\alpha_{1,\sigma}$ be two negatively ordered stated $n$-web diagrams such that $\alpha_1$ and $\alpha_{1,\sigma}$ are identical outside of $U_1$. Furthermore, $\alpha_1 \cap U_1$ (resp. $\alpha_{1,\sigma} \cap U_1$) corresponds to the diagram on the left-hand (resp. right-hand) side of relation \eqref{wzh.five}. Then we have $w(\alpha_{1,\sigma})=w(\alpha_1)$ and $e(\alpha_{1,\sigma}) = e(\alpha_1)+n$. Thus
	 \begin{align*}
	 	\varphi_{\epsilon}(a_v\sum_{\sigma\in \fS}(-q_v)^{\ell(\sigma)}\alpha_{1,\sigma})& = (-1)^{\lambda}
	 	a_u\sum_{\sigma\in \fS}(-q_u)^{\ell(\sigma)} \epsilon^{2w(\alpha_1) +\frac{e(\alpha_1)+n}{2}} \alpha_{1,\sigma}\\& = \epsilon^{2w(\alpha_1) +\frac{e(\alpha_1)}{2}}\alpha_1 = \varphi_{\epsilon}(\alpha_1).
	 \end{align*}
	 
	 Relation \eqref{wzh.six}: We have 
	 \begin{align*}
	 	\varphi_{\epsilon}(
	 	\raisebox{-.20in}{
	 		\begin{tikzpicture}
	 			\tikzset{->-/.style=
	 				{decoration={markings,mark=at position #1 with
	 						{\arrow{latex}}},postaction={decorate}}}
	 			\filldraw[draw=white,fill=gray!20] (-0.7,-0.7) rectangle (0,0.7);
	 			\draw [line width =1.5pt,decoration={markings, mark=at position 1 with {\arrow{>}}},postaction={decorate}](0,0.7)--(0,-0.7);
	 			\draw [color = black, line width =1pt] (0 ,0.3) arc (90:270:0.5 and 0.3);
	 			\node [right]  at(0,0.3) {$i$};
	 			\node [right] at(0,-0.3){$j$};
	 			\filldraw[fill=white,line width =0.8pt] (-0.5 ,0) circle (0.07);
	 	\end{tikzpicture}}) =\epsilon 
 	\raisebox{-.20in}{
 	\begin{tikzpicture}
 	\tikzset{->-/.style=
 		{decoration={markings,mark=at position #1 with
 				{\arrow{latex}}},postaction={decorate}}}
 	\filldraw[draw=white,fill=gray!20] (-0.7,-0.7) rectangle (0,0.7);
 	\draw [line width =1.5pt,decoration={markings, mark=at position 1 with {\arrow{>}}},postaction={decorate}](0,0.7)--(0,-0.7);
 	\draw [color = black, line width =1pt] (0 ,0.3) arc (90:270:0.5 and 0.3);
 	\node [right]  at(0,0.3) {$i$};
 	\node [right] at(0,-0.3){$j$};
 	\filldraw[fill=white,line width =0.8pt] (-0.5 ,0) circle (0.07);
 \end{tikzpicture}}   =  \delta_{\bar j,i }\, \epsilon\,  c_{i,u}\ \raisebox{-.20in}{
	 		\begin{tikzpicture}
	 			\tikzset{->-/.style=
	 				{decoration={markings,mark=at position #1 with
	 						{\arrow{latex}}},postaction={decorate}}}
	 			\filldraw[draw=white,fill=gray!20] (-0.7,-0.7) rectangle (0,0.7);
	 			\draw [line width =1.5pt](0,0.7)--(0,-0.7);
	 	\end{tikzpicture}}=\delta_{\bar j,i }\, c_{i,v}\ \raisebox{-.20in}{
	 	\begin{tikzpicture}
	 	\tikzset{->-/.style=
	 		{decoration={markings,mark=at position #1 with
	 				{\arrow{latex}}},postaction={decorate}}}
	 	\filldraw[draw=white,fill=gray!20] (-0.7,-0.7) rectangle (0,0.7);
	 	\draw [line width =1.5pt](0,0.7)--(0,-0.7);
 	\end{tikzpicture}}\,.
	 \end{align*}
	 
	 Relation \eqref{wzh.seven}:  Let $U_2$ be the open square in relation \eqref{wzh.seven}, and let $\alpha_2,\alpha_{2}'$ be two negatively ordered stated $n$-web diagrams such that $\alpha_2$ and $\alpha_{2}'$ are identical outside $U_2$ and $\alpha_2\cap U_2$ (resp. $\alpha_{2}'\cap U_2$) is like the picture on the left-hand (resp. right-hand) side of relation \eqref{wzh.seven}. Then we have $w(\alpha_2') = w(\alpha_{2})$ and $e(\alpha_{2}') = e(\alpha_2)+2$. Thus
	 \begin{align*}
	 	\varphi_{\epsilon}(\sum_{1\leq i\leq n}c_{\bar{i},v}^{-1}\alpha_2')
	 	=\sum_{1\leq i\leq n}\epsilon^{-1}c_{\bar{i},u}^{-1}\epsilon^{2w(\alpha_2)+\frac{e(\alpha_2)+2}{2}}\alpha_2' = \epsilon^{2w(\alpha_2)+\frac{e(\alpha_2)}{2}}\alpha_2 = \varphi_{\epsilon}(\alpha_2).
	 \end{align*}
	 
	 Relation \eqref{wzh.eight}:
	    Let $U_3$ be the open square in relation \eqref{wzh.eight}, and let $\alpha_3,\alpha_{3}',\alpha_3''$ be three negatively ordered stated $n$-web diagrams such that $\alpha_3,\alpha_{3}',\alpha_3''$ are identical to each other outside $U_3$ and $\alpha_3\cap U_3 = \raisebox{-.20in}{
	    	
	    	\begin{tikzpicture}
	    		\tikzset{->-/.style=
	    			
	    			{decoration={markings,mark=at position #1 with
	    					
	    					{\arrow{latex}}},postaction={decorate}}}
	    		\filldraw[draw=white,fill=gray!20] (-0,-0.2) rectangle (1, 1.2);
	    		\draw [line width =1.5pt,decoration={markings, mark=at position 1 with {\arrow{>}}},postaction={decorate}](1,1.2)--(1,-0.2);
	    		\draw [line width =1pt](0.6,0.6)--(1,1);
	    		\draw [line width =1pt](0.6,0.4)--(1,0);
	    		\draw[line width =1pt] (0,0)--(0.4,0.4);
	    		\draw[line width =1pt] (0,1)--(0.4,0.6);
	    		\draw[line width =1pt] (0.4,0.6)--(0.6,0.4);
	    		\filldraw[fill=white,line width =0.8pt] (0.2 ,0.2) circle (0.07);
	    		\filldraw[fill=white,line width =0.8pt] (0.2 ,0.8) circle (0.07);
	    		\node [right]  at(1,1) {$i$};
	    		\node [right]  at(1,0) {$j$};
	    	\end{tikzpicture}
	    }, \alpha_3'\cap U_3 = \raisebox{-.20in}{
	    
	    \begin{tikzpicture}
	    	\tikzset{->-/.style=
	    		
	    		{decoration={markings,mark=at position #1 with
	    				
	    				{\arrow{latex}}},postaction={decorate}}}
	    	\filldraw[draw=white,fill=gray!20] (-0,-0.2) rectangle (1, 1.2);
	    	\draw [line width =1.5pt,decoration={markings, mark=at position 1 with {\arrow{>}}},postaction={decorate}](1,1.2)--(1,-0.2);
	    	\draw [line width =1pt](0,0.8)--(1,0.8);
	    	\draw [line width =1pt](0,0.2)--(1,0.2);
	    	\filldraw[fill=white,line width =0.8pt] (0.2 ,0.8) circle (0.07);
	    	\filldraw[fill=white,line width =0.8pt] (0.2 ,0.2) circle (0.07);
	    	\node [right]  at(1,0.8) {$i$};
	    	\node [right]  at(1,0.2) {$j$};
	    \end{tikzpicture}
	    },\alpha_3''\cap U_3 = \raisebox{-.20in}{
	    
	    \begin{tikzpicture}
	    	\tikzset{->-/.style=
	    		
	    		{decoration={markings,mark=at position #1 with
	    				
	    				{\arrow{latex}}},postaction={decorate}}}
	    	\filldraw[draw=white,fill=gray!20] (-0,-0.2) rectangle (1, 1.2);
	    	\draw [line width =1.5pt,decoration={markings, mark=at position 1 with {\arrow{>}}},postaction={decorate}](1,1.2)--(1,-0.2);
	    	\draw [line width =1pt](0,0.8)--(1,0.8);
	    	\draw [line width =1pt](0,0.2)--(1,0.2);
	    	\filldraw[fill=white,line width =0.8pt] (0.2 ,0.8) circle (0.07);
	    	\filldraw[fill=white,line width =0.8pt] (0.2 ,0.2) circle (0.07);
	    	\node [right]  at(1,0.8) {$j$};
	    	\node [right]  at(1,0.2) {$i$};
	    \end{tikzpicture}
	    }$. Then we have $w(\alpha_3') = w(\alpha_{3}'') = w(\alpha_3)-1$ and $e(\alpha_{3}') = e(\alpha_3'') =e(\alpha_3)$. Thus 
	    \begin{align*}
	    	\varphi_{\epsilon}\big(  q_v^{-\frac{1}{n}}(\delta_{j<i}(q_v-q_v^{-1})\alpha_3'+q_{v}^{\delta_{i,j}}\alpha_3'')  \big)
	    	=&\epsilon^2\epsilon^{2(w(\alpha_3)-1)+\frac{e(\alpha_3)}{2}} q_u^{-\frac{1}{n}}(\delta_{j<i}(q_u-q_u^{-1})\alpha_3'+q_{u}^{\delta_{i,j}}\alpha_3'')\\
	    	=&\epsilon^{2w(\alpha_3)+\frac{e(\alpha_3)}{2}}\alpha_3 = \varphi_{\epsilon}(\alpha_3).
	    \end{align*}
	 
	In the following of the proof, we will still use $\alpha_1,\alpha_{1,\sigma},\alpha_{2},\alpha_2',\alpha_{3},\alpha_3',\alpha_3''$ to denote the corresponding negatively ordered stated $n$-web diagrams as in  Case 1.
	
	\text{\bf Case 2}: $k$ is odd, $m$ is even. Suppose $m = 2l$. From direct calculations, we have $\epsilon^n = -1$,  $c_{i,u} = -\epsilon^{-1}c_{i,v}$, and $a_u = (-\epsilon)^{kl}a_v$.
	
	Relation \eqref{wzh.five}: We have 
	\begin{align*}
		\varphi_{\epsilon}(a_v\sum_{\sigma\in \fS}(-q_v)^{\ell(\sigma)}\alpha_{1,\sigma})& = (-\epsilon^{-1})^{kl}
		a_u\sum_{\sigma\in \fS}(-q_u)^{\ell(\sigma)} (-\epsilon)^{2w(\alpha_1) +\frac{e(\alpha_1)+n}{2}} \alpha_{1,\sigma}\\& = (-\epsilon)^{2w(\alpha_1) +\frac{e(\alpha_1)}{2}}\alpha_1 = \varphi_{\epsilon}(\alpha_1).
	\end{align*}
	
	Relation \eqref{wzh.six}: We have 
	\begin{align*}
		\varphi_{\epsilon}(
		\raisebox{-.20in}{
			\begin{tikzpicture}
				\tikzset{->-/.style=
					{decoration={markings,mark=at position #1 with
							{\arrow{latex}}},postaction={decorate}}}
				\filldraw[draw=white,fill=gray!20] (-0.7,-0.7) rectangle (0,0.7);
				\draw [line width =1.5pt,decoration={markings, mark=at position 1 with {\arrow{>}}},postaction={decorate}](0,0.7)--(0,-0.7);
				\draw [color = black, line width =1pt] (0 ,0.3) arc (90:270:0.5 and 0.3);
				\node [right]  at(0,0.3) {$i$};
				\node [right] at(0,-0.3){$j$};
				\filldraw[fill=white,line width =0.8pt] (-0.5 ,0) circle (0.07);
		\end{tikzpicture}}) =-\epsilon 
		\raisebox{-.20in}{
			\begin{tikzpicture}
				\tikzset{->-/.style=
					{decoration={markings,mark=at position #1 with
							{\arrow{latex}}},postaction={decorate}}}
				\filldraw[draw=white,fill=gray!20] (-0.7,-0.7) rectangle (0,0.7);
				\draw [line width =1.5pt,decoration={markings, mark=at position 1 with {\arrow{>}}},postaction={decorate}](0,0.7)--(0,-0.7);
				\draw [color = black, line width =1pt] (0 ,0.3) arc (90:270:0.5 and 0.3);
				\node [right]  at(0,0.3) {$i$};
				\node [right] at(0,-0.3){$j$};
				\filldraw[fill=white,line width =0.8pt] (-0.5 ,0) circle (0.07);
		\end{tikzpicture}}   =  \delta_{\bar j,i }\, (-\epsilon)\,  c_{i,u}\ \raisebox{-.20in}{
			\begin{tikzpicture}
				\tikzset{->-/.style=
					{decoration={markings,mark=at position #1 with
							{\arrow{latex}}},postaction={decorate}}}
				\filldraw[draw=white,fill=gray!20] (-0.7,-0.7) rectangle (0,0.7);
				\draw [line width =1.5pt](0,0.7)--(0,-0.7);
		\end{tikzpicture}}=\delta_{\bar j,i }\, c_{i,v}\ \raisebox{-.20in}{
			\begin{tikzpicture}
				\tikzset{->-/.style=
					{decoration={markings,mark=at position #1 with
							{\arrow{latex}}},postaction={decorate}}}
				\filldraw[draw=white,fill=gray!20] (-0.7,-0.7) rectangle (0,0.7);
				\draw [line width =1.5pt](0,0.7)--(0,-0.7);
		\end{tikzpicture}}\,.
	\end{align*}
	
	Relation \eqref{wzh.seven}: We have \begin{align*}
		\varphi_{\epsilon}(\sum_{1\leq i\leq n}c_{\bar{i},v}^{-1}\alpha_2')
		=\sum_{1\leq i\leq n}-\epsilon^{-1}c_{\bar{i},u}^{-1}(-\epsilon)^{2w(\alpha_2)+\frac{e(\alpha_2)+2}{2}}\alpha_2' = (-\epsilon)^{2w(\alpha_2)+\frac{e(\alpha_2)}{2}}\alpha_2 = \varphi_{\epsilon}(\alpha_2).
	\end{align*}
	
	Relation \eqref{wzh.eight}: We have \begin{align*}
		\varphi_{\epsilon}\big(  q_v^{-\frac{1}{n}}(\delta_{j<i}(q_v-q_v^{-1})\alpha_3'+q_{v}^{\delta_{i,j}}\alpha_3'')  \big)
		=&\epsilon^2 (-\epsilon)^{2(w(\alpha_3)-1)+\frac{e(\alpha_3)}{2}} q_u^{-\frac{1}{n}}(\delta_{j<i}(q_u-q_u^{-1})\alpha_3'+q_{u}^{\delta_{i,j}}\alpha_3'')\\
		=&(-\epsilon)^{2w(\alpha_3)+\frac{e(\alpha_3)}{2}}\alpha_3 = \varphi_{\epsilon}(\alpha_3).
	\end{align*}
	
	\text{\bf Case 3}: $k$ is odd, $m$ is odd, and $\frac{k+m}{2}$ is odd. 
	 From direct calculations, we have $\epsilon^n = -1$, $c_{i,u} = -\epsilon^{-1}c_{i,v}$, and $a_u = a_v$. The proofs for verifying relations \eqref{wzh.six}–\eqref{wzh.eight} are similar to the proof for Case 2.
	 
	 Relation \eqref{wzh.five}: We have $  w(\alpha_{1,\sigma})=w(\alpha_1)$, and $t(\alpha_{1,\sigma}) = t(\alpha_1)$ or $t(\alpha_1)+n$. We always have $(-\epsilon)^{t(\alpha_{1,\sigma})} = (-\epsilon)^{t(\alpha_{1})}$ because $(-\epsilon)^n = 1$. Then it is a trivial check that $\varphi_{\epsilon}$ preserves relation \eqref{wzh.five}.
	
		\text{\bf Case 4}: $k$ is odd, $m$ is odd, and $\frac{k+m}{2}$ is even. 
		 From direct calculations, we have $\epsilon^n = -1$, $c_{i,u} = -\epsilon^{-1}c_{i,v}$, and $a_u = - a_v$. The proof  is similar with the proof for Case 3.  

\end{proof}

The following theorem is analogous to Theorem \ref{stated1}. We establish an $R$-linear isomorphism from $S_n(\Sigma, v)$ to $S_n(\Sigma, \epsilon v)$ under the conditions $\epsilon^m = 1$ and $m \mid n$. Similar to the $R$-linear isomorphism presented in Theorem \ref{stated1}, the corresponding isomorphism in this theorem also defines an algebra isomorphism from $S_n(\Sigma, v)_m$ to $S_n(\Sigma, \epsilon v)_m$.



\begin{theorem}\label{stated2}
	Let $\Sigma$ be a pb surface, let $m,n$ be two positive integers with $m\mid n$, and let $\epsilon\in R$ with $\epsilon^{m} = 1$. Then there exists an $R$-linear isomorphism $\varphi_{\epsilon}:S_n(\Sigma,v)\rightarrow S_n(\Sigma,\epsilon v)$, defined by
	\begin{equation*}
		\varphi_{\epsilon}(\alpha)=
		\begin{cases}
			\epsilon^{2w(\alpha)+\frac{e(\alpha)}{2}} \alpha& n \text{ is even} \\
			\epsilon^{2w(\alpha)+t(\alpha)} \alpha& n \text{ is odd,}
		\end{cases}
	\end{equation*}
	where $\alpha\in S_n(\Sigma,v)$ is a negatively ordered stated $n$-web diagram.  In particular $\varphi_{\epsilon}(\alpha) = \epsilon^{2(w(\alpha))}$ if $\alpha$ is a negatively ordered stated $n$-web diagram without endpoints.
\end{theorem}
\begin{proof}
	Set $u=\epsilon v$.
	From direct calculations, we have $\epsilon^n = 1$, $q_u^{\frac{1}{2}} = q_v^{\frac{1}{2}}$,  and $c_{i,u} = \epsilon^{-1}c_{i,v}$ for $1\leq i\leq n$. Note that $a_u = \epsilon^{\frac{n}{2}}a_v$ when $n$ is even and $a_u = a_v$ when $n$ is odd. Then, using the same technique as in Theorem \ref{stated1}, we can prove that $\varphi_{\epsilon}$ is a well-defined linear map. Trivially $\varphi_{\epsilon^{-1}}:S_n(\Sigma,\epsilon v)\rightarrow S_n(\Sigma,v)$ is the inverse of $\varphi_{\epsilon}$.

\end{proof}

The following Theorem combines Theorems \ref{stated1} and \ref{stated2}.

\begin{theorem}\label{c_iso}
	Let $\Sigma$ be a pb surface, let $k,m,n$ be three positive integers with $n = km$, and let $\epsilon\in R$ with $\epsilon^{2m} = 1$. Then there exists an $R$-linear isomorphism $\varphi_{\epsilon}:S_n(\Sigma,v)\rightarrow S_n(\Sigma,\epsilon v)$, defined by
	\begin{equation*}
		\varphi_{\epsilon}(\alpha)=
		\begin{cases}
			\epsilon^{2w(\alpha)+\frac{e(\alpha)}{2}} \alpha& \epsilon^m =-1,\text{ } k \text{ is even} \\
			(-\epsilon)^{2w(\alpha)+\frac{e(\alpha)}{2}} \alpha&\epsilon^m =-1,\text{ } k \text{ is odd, and }m\text{ is even}\\
			(-\epsilon)^{2w(\alpha)+t(\alpha)} \alpha&\epsilon^m =-1, \text{ }k \text{ is odd, }m\text{ is odd, and }\frac{k+m}{2} \text{ is odd}\\
			(-1)^{p(\alpha)}(-\epsilon)^{2w(\alpha)+t(\alpha)} \alpha&\epsilon^m =-1, \text{ }k \text{ is odd, }m\text{ is odd, and }\frac{k+m}{2} \text{ is even}\\
			\epsilon^{2w(\alpha)+\frac{e(\alpha)}{2}} \alpha&\epsilon^m =1,\text{ } n \text{ is even} \\
			\epsilon^{2w(\alpha)+t(\alpha)} \alpha&\epsilon^m =1, \text{ }n \text{ is odd,}
		\end{cases}
	\end{equation*}
	where $\alpha\in S_n(\Sigma,v)$ is a negatively ordered stated $n$-web diagram.  In particular $\varphi_{\epsilon}(\alpha) = \epsilon^{2(w(\alpha))}$ if $\alpha$ is a negatively ordered stated $n$-web diagram without endpoints.
\end{theorem}
\begin{proof}
	Theorems \ref{stated1}, \ref{stated2}.
\end{proof}

\begin{corollary}
Let $\Sigma$ be a pb surface, let $m,n$ be two positive integers with $m\mid n$, and let $\epsilon\in R$ with $\epsilon^{2m} = 1$.
 Then the  $R$-linear isomorphism $\varphi_{\epsilon}:S_n(\Sigma,v)\rightarrow S_n(\Sigma, \epsilon v)$ restricts to an algebra isomorphism $\varphi_{\epsilon}|_{S_n(\Sigma,v)_m}:S_n(\Sigma,v)_m\rightarrow S_n(\Sigma,\epsilon v)_m$.
\end{corollary}
\begin{proof}
	The proof is similar with Corollary \ref{skein_cor}. 
\end{proof}


%
%
%
%
%
%
%
%

\section{Spin structure and the $SL_n$-skein space}

The $SU_n$-skein module of an oriented 3-manifold is defined in \cite{sikora2005skein}. The $SU_n$-skein module of the 3-manifold $M$ is built of based $n$-webs which are defined as the $n$-webs in $M$, except that the half-edges incident  to any of their $n$-valent vertices are linearly ordered. The $SU_n$-skein module, denoted as $S_n^{b}(M,v^2)$, is the quotient of the $R$-module freely generated by the set of isotopy classes of based $n$-webs subject to relations \eqref{w.cross}-\eqref{w.unknot} and \eqref{wzh.nfour}. L{\^e} and Sikora proved the equivalence between the $SU_n$-skein theory and the $SL_n$-skein theory \cite{le2021stated}. To simplify the proof, we will work with the $SU_n$-skein theory in this section (all the discussions can be pulled back to the $SL_n$-skein theory using the equivalence built by L{\^e} and Sikora). 

For a 3-manifold $M$,  Barrett established an isomorphism from $S_2^{b}(M,v^2)$ to $S_2^{b}(M,-v^2)$ using a spin structure of $M$ \cite{barrett1999skein}. In this section, we will generalize Barrett's work to all $n$.

\beq\label{wzh.nfour}
\raisebox{-.30in}{
	\begin{tikzpicture}
		\tikzset{->-/.style=
			{decoration={markings,mark=at position #1 with
					{\arrow{latex}}},postaction={decorate}}}
		\filldraw[draw=white,fill=gray!20] (-1,-0.7) rectangle (1.2,1.3);
		\draw [line width =1pt,decoration={markings, mark=at position 0.5 with {\arrow{>}}},postaction={decorate}](-1,1)--(0,0);
		\draw [line width =1pt,decoration={markings, mark=at position 0.5 with {\arrow{>}}},postaction={decorate}](-1,0)--(0,0);
		\draw [line width =1pt,decoration={markings, mark=at position 0.5 with {\arrow{>}}},postaction={decorate}](-1,-0.4)--(0,0);
		\draw [line width =1pt,decoration={markings, mark=at position 0.5 with {\arrow{<}}},postaction={decorate}](1.2,1)  --(0.2,0);
		\draw [line width =1pt,decoration={markings, mark=at position 0.5 with {\arrow{<}}},postaction={decorate}](1.2,0)  --(0.2,0);
		\draw [line width =1pt,decoration={markings, mark=at position 0.5 with {\arrow{<}}},postaction={decorate}](1.2,-0.4)--(0.2,0);
		\node  at(-0.8,0.5) {$\vdots$};
		\node  at(1,0.5) {$\vdots$};
\end{tikzpicture}}=q_v^{n(n-1)}\cdot \sum_{\sigma\in \fS}
(-q_v^{\frac{1-n}n})^{\ell(\sigma)} \raisebox{-.30in}{
	\begin{tikzpicture}
		\tikzset{->-/.style=
			{decoration={markings,mark=at position #1 with
					{\arrow{latex}}},postaction={decorate}}}
		\filldraw[draw=white,fill=gray!20] (-1,-0.7) rectangle (1.2,1.3);
		\draw [line width =1pt,decoration={markings, mark=at position 0.5 with {\arrow{>}}},postaction={decorate}](-1,1)--(0,0);
		\draw [line width =1pt,decoration={markings, mark=at position 0.5 with {\arrow{>}}},postaction={decorate}](-1,0)--(0,0);
		\draw [line width =1pt,decoration={markings, mark=at position 0.5 with {\arrow{>}}},postaction={decorate}](-1,-0.4)--(0,0);
		\draw [line width =1pt,decoration={markings, mark=at position 0.5 with {\arrow{<}}},postaction={decorate}](1.2,1)  --(0.2,0);
		\draw [line width =1pt,decoration={markings, mark=at position 0.5 with {\arrow{<}}},postaction={decorate}](1.2,0)  --(0.2,0);
		\draw [line width =1pt,decoration={markings, mark=at position 0.5 with {\arrow{<}}},postaction={decorate}](1.2,-0.4)--(0.2,0);
		\node  at(-0.8,0.5) {$\vdots$};
		\node  at(1,0.5) {$\vdots$};
		\filldraw[draw=black,fill=gray!20,line width =1pt]  (0.1,0.3) ellipse (0.4 and 0.7);
		\node  at(0.1,0.3){$\sigma_{+}$};
\end{tikzpicture}}
\eeq

Let $\alpha$ be a based $n$-web in $M$, and $H$ be an isotopy of $\alpha$. We use $\alpha^H$ to denote the based $n$-web in $M$ obtained from $\alpha$ by doing isotopy $H$. We use $K(\alpha)$ to denote the number of knot components of $\alpha$.

\begin{rem}\label{perm}

	Suppose that $n$ is a positive even number. 
	Consider a based $n$-web $\alpha$ in a 3-manifold $M$ with $k$ sinks and $k$ sources ($k > 0$), where the sinks are labeled from $1$ to $k$ and the sources are also labeled from $1$ to $k$.
 Through an isotopy $H$, we deform $\alpha$ such that, for each $1 \leq i \leq k$, the sink and source labeled by $i$ lie in a small open cube $U_i$. The intersection $\alpha^H \cap U_i$ resembles the illustration on the left side of relation \eqref{wzh.nfour}.

	For $\sigma_{1},\sigma_2,\dots,\sigma_k \in \fS$, denote $\alpha_{\sigma_1,\cdots,\sigma_k}^H$ as the based $n$-web obtained from $\alpha^H$. This web coincides with $\alpha^H$ outside $\cup_{1\leq i\leq k}U_i$, and within each $U_i$, the intersection $\alpha_{\sigma_1,\cdots,\sigma_k}^H \cap U_i$ follows the pattern on the right side of relation \eqref{wzh.nfour}, with $\sigma_+ = (\sigma_i)_+$. In $S_n^{b}(M,v^2)$, relation \eqref{wzh.nfour} implies 
$$\alpha = q_v^{n(n-1)k}\sum_{\sigma_1,\dots,\sigma_k\in \fS}(-q_v^{\frac{1-n}{n}})^{\ell(\sigma_1)+\cdots+ \ell(\sigma_k)} \alpha_{\sigma_1,\cdots,\sigma_k}^H.$$

For each $1\leq i\leq k$, we know that there is a linear order on all half edges incident to the  $i$-th sink (and source) of $\alpha$. Then we color the $j$-th half edge incident to the $i$-sink (resp. the source)  with $(i-1)n+j$ (resp.  $(i-1)n+n+1-j$). Then there exists $\tau_{\alpha}\in \mathfrak{S}_{kn}$ such that, for each $1\leq i\leq kn$, the starting half edge colored by $i$ is connected to the targeting half edge
colored by $\tau_{\alpha}(i)$.

Although $\tau_{\alpha}$ is dependent on how we label the sinks and sources of $\alpha$, the sign  $(-1)^{\ell(\tau_{\alpha})}$ is irrelevant of the labelings of sinks and sources because $n$ is even.  It is easy to show 
$$(-1)^{K(\alpha_{\sigma_1,\cdots,\sigma_k}^H)} = (-1)^{K(\alpha)+\ell(\tau_{\alpha})+\ell(\sigma_1)+\cdots+\ell(\sigma_k)}.$$
\end{rem}

Suppose $n$ is even and $s$ is a spin structure of $M$. For any based $n$-web $\alpha$ in $M$, L{\^e} and Sikora defined $s(\alpha) \in \mathbb{Z}_2$ on page 76 of \cite{le2021stated}, with the properties that $s(\text{trivial knot}) = 1$ and $s(\alpha') = s(\alpha) + 1$ if $\alpha'$ differs from $\alpha$ by a single positive or negative kink.

The following Theorem generalizes Barrett's work \cite{barrett1999skein} to all $n$.

\begin{theorem}
Let $M$ be a 3-manifold, and let $s$ be a spin structure of $M$. For each positive integer $n$, there exists a linear isomorphism $\Phi_{n}:S_n^{b}(M,v^2)\rightarrow S_n^{b}(M,-v^2)$, defined by  
\begin{equation*}
	\Phi_n(\alpha)=
	\begin{cases}
		\alpha& n \text{ is odd} \\
		(-1)^{s(\alpha)+K(\alpha)+\ell(\tau_{\alpha})}& n \text{ is even,}
	\end{cases}
\end{equation*}
where $\alpha$ is a based $n$-web in $M$, the permutation $\tau_{\alpha} = (1)$ if $\alpha$ has no sinks or sources and $\tau_{\alpha}$ is the permutation  defined as in Remark \ref{perm} if $\alpha$ has sinks and sources. In particular $\Phi_n$ is an algebra isomorphism when $M$ is the thickening of an oriented surface.
\end{theorem}
\begin{proof}
When $n$ is odd, relations \eqref{w.cross}-\eqref{w.unknot} and \eqref{wzh.nfour} for $S_n^{b}(M,v^2)$ are the same with the corresponding relations for $S_n^{b}(M,-v^2)$. 

Suppose $n$ is even.
Obviously, the map $\Phi_n$ is well-defined on the set of isotopy classes of based $n$-webs.

 Suppose that the based $n$-web $\alpha$ in $M$ has $k$ sinks and $k$ sources with $k>0$. We do the same procedures to $\alpha$ and follow the same notations as in Remark \ref{perm}. In $S_n^{b}(M,-v^2)$, we have 
\begin{equation}\label{eeeq}
\begin{split}
\Phi_n(\alpha) = &(-1)^{s(\alpha)+K(\alpha)+\ell(\tau_{\alpha})}\alpha\\
=&(-1)^{s(\alpha)+K(\alpha)+\ell(\tau_{\alpha})}
q_v^{n(n-1)k}\sum_{\sigma_1,\dots,\sigma_k\in \fS}(q_v^{\frac{1-n}{n}})^{\ell(\sigma_1)+\cdots+ \ell(\sigma_k)} \alpha_{\sigma_1,\cdots,\sigma_k}^H\\
=&(-1)^{s(\alpha)}
q_v^{n(n-1)k}\sum_{\sigma_1,\dots,\sigma_k\in \fS}(-q_v^{\frac{1-n}{n}})^{\ell(\sigma_1)+\cdots+ \ell(\sigma_k)}(-1)^{K(\alpha)+\ell(\tau_{\alpha})+\ell(\sigma_1)+\cdots+ \ell(\sigma_k)} \alpha_{\sigma_1,\cdots,\sigma_k}^H\\
=&(-1)^{s(\alpha)}
q_v^{n(n-1)k}\sum_{\sigma_1,\dots,\sigma_k\in \fS}(-q_v^{\frac{1-n}{n}})^{\ell(\sigma_1)+\cdots+ \ell(\sigma_k)}(-1)^{K(\alpha_{\sigma_1,\cdots,\sigma_k}^H)} \alpha_{\sigma_1,\cdots,\sigma_k}^H.
\end{split}
\end{equation}

Then we are trying to show $\Phi_n$ preserves relations \eqref{w.cross}-\eqref{w.unknot} and \eqref{wzh.nfour}.

Relation \eqref{w.cross}: We use $U$ to denote the small open cube in relation \eqref{w.cross}. Let $\alpha_1,\alpha_2,\alpha_3$ be three based $n$-webs in $M$ such that they are identical to each other outside $U$ and 
$\alpha_1\cap U=
\raisebox{-.20in}{

\begin{tikzpicture}
\tikzset{->-/.style=

{decoration={markings,mark=at position #1 with

{\arrow{latex}}},postaction={decorate}}}
\filldraw[draw=white,fill=gray!20] (-0,-0.2) rectangle (1, 1.2);
\draw [line width =1pt,decoration={markings, mark=at position 0.5 with {\arrow{>}}},postaction={decorate}](0.6,0.6)--(1,1);
\draw [line width =1pt,decoration={markings, mark=at position 0.5 with {\arrow{>}}},postaction={decorate}](0.6,0.4)--(1,0);
\draw[line width =1pt] (0,0)--(0.4,0.4);
\draw[line width =1pt] (0,1)--(0.4,0.6);
\draw[line width =1pt] (0.4,0.6)--(0.6,0.4);
\end{tikzpicture}},
\alpha_2\cap U =
\raisebox{-.20in}{
\begin{tikzpicture}
\tikzset{->-/.style=

{decoration={markings,mark=at position #1 with

{\arrow{latex}}},postaction={decorate}}}
\filldraw[draw=white,fill=gray!20] (-0,-0.2) rectangle (1, 1.2);
\draw [line width =1pt,decoration={markings, mark=at position 0.5 with {\arrow{>}}},postaction={decorate}](0.6,0.6)--(1,1);
\draw [line width =1pt,decoration={markings, mark=at position 0.5 with {\arrow{>}}},postaction={decorate}](0.6,0.4)--(1,0);
\draw[line width =1pt] (0,0)--(0.4,0.4);
\draw[line width =1pt] (0,1)--(0.4,0.6);
\draw[line width =1pt] (0.6,0.6)--(0.4,0.4);
\end{tikzpicture}
},\alpha_3\cap U=
\raisebox{-.20in}{

\begin{tikzpicture}
\tikzset{->-/.style=

{decoration={markings,mark=at position #1 with

{\arrow{latex}}},postaction={decorate}}}
\filldraw[draw=white,fill=gray!20] (-0,-0.2) rectangle (1, 1.2);
\draw [line width =1pt,decoration={markings, mark=at position 0.5 with {\arrow{>}}},postaction={decorate}](0,0.8)--(1,0.8);
\draw [line width =1pt,decoration={markings, mark=at position 0.5 with {\arrow{>}}},postaction={decorate}](0,0.2)--(1,0.2);
\end{tikzpicture}
}$. It is easy to show $s(\alpha_1) = s(\alpha_2) = s(\alpha_3)$. If  $\alpha_1,\alpha_2,\alpha_3$ contain no sinks or sources, we have $(-1)^{K(\alpha_1)} = (-1)^{K(\alpha_2)} = -(-1)^{K(\alpha_3)}$.
Then, in $S_n^{b}(M,-v^2)$, we have 
\begin{align*}
&\Phi_n(q_v^{\frac{1}{n}}\alpha_1 - q_v^{-\frac{1}{n}}\alpha_2) = 
q_v^{\frac{1}{n}}(-1)^{s(\alpha_1)+K(\alpha_1)}\alpha_1 - q_v^{-\frac{1}{n}}(-1)^{s(\alpha_2)+K(\alpha_2)}\alpha_2\\
=&(-1)^{s(\alpha_3)+K(\alpha_3)}(q_v^{-\frac{1}{n}}\alpha_2-q_v^{\frac{1}{n}}\alpha_1)=(-1)^{s(\alpha_3)+K(\alpha_3)}(q_v - q_v^{-1})\alpha_3 = \Phi_n((q_v - q_v^{-1})\alpha_3).
\end{align*}

If $\alpha_1$, $\alpha_2$, and $\alpha_3$ contain sinks and sources, we label the sinks and sources consistently across $\alpha_1$, $\alpha_2$, and $\alpha_3$. We can use the same isotopy $H$ to pair the sinks and sources for $\alpha_1,\alpha_2,\alpha_3$ such that the support of $H$ is contained in $M\setminus U$. Let $\sigma_1,\cdots,\sigma_k\in \fS$. We have that
  $(\alpha_1)_{\sigma_1,\cdots,\sigma_k}^{H},(\alpha_2)_{\sigma_1,\cdots,\sigma_k}^{H},(\alpha_3)_{\sigma_1,\cdots,\sigma_k}^{H}$ are identical to each other outside $U$, and
$$(\alpha_1)_{\sigma_1,\cdots,\sigma_k}^{H}\cap U =\alpha_1\cap U,\;(\alpha_2)_{\sigma_1,\cdots,\sigma_k}^{H}\cap U =\alpha_2\cap U,\;(\alpha_3)_{\sigma_1,\cdots,\sigma_k}^{H}\cap U =\alpha_3\cap U.$$ Futhermore, 
$(-1)^{K((\alpha_1)_{\sigma_1,\cdots,\sigma_k}^{H})} = (-1)^{K((\alpha_2)_{\sigma_1,\cdots,\sigma_k}^{H})} = - (-1)^{K((\alpha_3)_{\sigma_1,\cdots,\sigma_k}^{H})}$. Note that, in $S_n^{b}(M,-v^2)$, we have 
$$q_v^{-\frac{1}{n}}(\alpha_2)_{\sigma_1,\cdots,\sigma_k}^{H} - q_v^{\frac{1}{n}} (\alpha_1)_{\sigma_1,\cdots,\sigma_k}^{H} =(q_v -q_v^{-1}) (\alpha_3)_{\sigma_1,\cdots,\sigma_k}^{H}.$$
From equation \eqref{eeeq}, we have 
\begin{align*}
&\Phi_n(q_v^{\frac{1}{n}}\alpha_1 - q_v^{-\frac{1}{n}}\alpha_2)\\=
&(-1)^{s(\alpha_1)}
q_v^{n(n-1)k}\sum_{\sigma_1,\dots,\sigma_k\in \fS}(-q_v^{\frac{1-n}{n}})^{\ell(\sigma_1)+\cdots+ \ell(\sigma_k)}(-1)^{K((\alpha_1)_{\sigma_1,\cdots,\sigma_k}^H)} q_v^{\frac{1}{n}} (\alpha_1)_{\sigma_1,\cdots,\sigma_k}^H\\
&-(-1)^{s(\alpha_2)}
q_v^{n(n-1)k}\sum_{\sigma_1,\dots,\sigma_k\in \fS}(-q_v^{\frac{1-n}{n}})^{\ell(\sigma_1)+\cdots+ \ell(\sigma_k)}(-1)^{K((\alpha_2)_{\sigma_1,\cdots,\sigma_k}^H)} q_v^{-\frac{1}{n}} (\alpha_2)_{\sigma_1,\cdots,\sigma_k}^H\\=
&(-1)^{s(\alpha_3)}
q_v^{n(n-1)k}\sum_{\sigma_1,\dots,\sigma_k\in \fS}(-q_v^{\frac{1-n}{n}})^{\ell(\sigma_1)+\cdots+ \ell(\sigma_k)}(-1)^{K((\alpha_3)_{\sigma_1,\cdots,\sigma_k}^H)} (q_v^{-\frac{1}{n}} (\alpha_2)_{\sigma_1,\cdots,\sigma_k}^H-q_v^{\frac{1}{n}} (\alpha_1)_{\sigma_1,\cdots,\sigma_k}^H )\\
=&(q_v-q_v^{-1})(-1)^{s(\alpha_3)}
q_v^{n(n-1)k}\sum_{\sigma_1,\dots,\sigma_k\in \fS}(-q_v^{\frac{1-n}{n}})^{\ell(\sigma_1)+\cdots+ \ell(\sigma_k)}(-1)^{K((\alpha_3)_{\sigma_1,\cdots,\sigma_k}^H)}  (\alpha_3)_{\sigma_1,\cdots,\sigma_k}^H\\
=&\Phi_n((q_v-q_v^{-1})\alpha_3).
\end{align*}

It is trivial that $\Phi_n$ preserves relations \eqref{w.twist} and \eqref{w.unknot}.

Relation \eqref{wzh.nfour}: In relation \eqref{wzh.nfour}, the open cube is denoted by $V$. Consider two based $n$-webs in $M$, namely $\beta$ and $\beta_{\sigma}$, with the property that they are identical outside $V$. The intersections $\beta \cap V$ and $\beta_{\sigma} \cap V$ correspond to pictures on the left-hand and right-hand sides of equation \eqref{wzh.nfour}, respectively.

If $\beta_{\sigma}$ has no sinks or sources, equation \eqref{eeeq} implies 
\begin{align*}
\Phi_n(\beta) &= (-1)^{s(\beta)}q_v^{n(n-1)}\sum_{\sigma\in \fS}(-1)^{K(\beta_{\sigma})} (-q_{v}^{\frac{1-n}{n}})^{\ell(\sigma)}\beta_{\sigma}\\=&
q_v^{n(n-1)}\sum_{\sigma\in \fS}(-1)^{K(\beta_{\sigma})+s(\beta_{\sigma})} (-q_{v}^{\frac{1-n}{n}})^{\ell(\sigma)}\beta_{\sigma}
\\=&\Phi_n\big(
q_v^{n(n-1)}\sum_{\sigma\in \fS}(-q_{v}^{\frac{1-n}{n}})^{\ell(\sigma)}\beta_{\sigma}\big).
\end{align*}

Suppose that $\beta_{\sigma}$ has $k$ sinks and $k$ sources with $k > 0$. We label the sinks of $\beta_{\sigma}$ from $1$ to $k$ and the sources of $\beta_{\sigma}$ from $1$ to $k$.
 We label the sink (resp. source) of $\beta$ shown in the relation \eqref{wzh.nfour} by $1$, and label the rest of the sinks (resp. sources) of $\beta$ from $2$ to $k+1$ such that, for each $1\leq i\leq k$, the  sink (resp. source) of $\beta$ labeled by $i+1$ is the sink (resp. source) of $\beta_{\sigma}$ labeled by $i$. Since the sink and the source of $\beta$ labeled by $1$ are already paired together, there exists an isotopy $F$ of $\beta$ such that the support of $F$ is contained in $M\setminus V$ and, for each $2\leq i\leq k+1$,  the sink and the source of $\beta^{F}$ labeled by $i$ are like the picture on the left-hand side of relation \eqref{wzh.nfour}. Obviously, we can regard $F$ as an isotopy of $\beta_{\sigma}$, and, for each $1\leq i\leq k$,  the sink and the source of $(\beta_{\sigma})^{F}$ labeled by $i$ are like the picture on the left-hand side of relation \eqref{wzh.nfour}. Then we have 
\begin{align*}
&\Phi_n\big(q_v^{n(n-1)}\sum_{\sigma\in \fS}(-q_v^{\frac{1-n}{n}})^{\ell(\sigma)}\beta_{\sigma}\big)\\
=&q_v^{n(n-1)}\sum_{\sigma\in \fS}(-q_v^{\frac{1-n}{n}})^{\ell(\sigma)}
(-1)^{s(\beta_{\sigma})}
q_v^{n(n-1)k}\sum_{\sigma_1,\dots,\sigma_k\in \fS}(-q_v^{\frac{1-n}{n}})^{\ell(\sigma_1)+\cdots+ \ell(\sigma_k)}(-1)^{K((\beta_{\sigma})_{\sigma_1,\cdots,\sigma_k}^F)} (\beta_{\sigma})_{\sigma_1,\cdots,\sigma_k}^F\\
=&q_v^{n(n-1)(k+1)}
(-1)^{s(\beta)}
\sum_{\sigma,\sigma_1,\dots,\sigma_k\in \fS}(-q_v^{\frac{1-n}{n}})^{\ell(\sigma)+\ell(\sigma_1)+\cdots+ \ell(\sigma_k)}(-1)^{K(\beta_{\sigma,\sigma_1,\cdots,\sigma_k}^F)} \beta_{\sigma,\sigma_1,\cdots,\sigma_k}^F\\
=&\Phi_n(\beta).
\end{align*}
\end{proof}

\def \MN {(M,\mathcal{N})}
\def \M {M,\mathcal{N}}

\section{The $R$-linear isomorphism between $S_n(M,\mathcal{N},v)$ and $S_n(M,\mathcal{N},-v)$}

For a marked 3-manifold $\MN$ and a positive integer $n$, we will show that there exists an $R$-linear isomorphism $\Psi_n: S_n(\M,v)\rightarrow S_n(\M,-v)$ such that it is an algebra isomorphism when $\MN$ is the thickening of a pb surface and $\Psi_n(\alpha) = \alpha$ for any $n$-web $\alpha$ in $\MN$ without endpoints.

Let $u = -v$. We have $q_u^{\frac{1}{n}} = u^2 = v^2 = q_v^{\frac{1}{n}}, q_u = u^{2n} = v^{2n} = q_v$. Then relations \eqref{w.cross}-\eqref{wzh.four} for $S_n(\M,u)$ are exactly the same with these relations for $S_n(\M,v)$. 


\begin{theorem}\label{5.1}
	Let $\MN$ be a marked 3-manifold, 
	and let $n$ be a positive integer.
	There exists an $R$-linear isomorphism $\Psi_n:S_n(\M,v)\rightarrow S_n(\M,-v)$ defined by
	\begin{equation*}
		\Psi_n(\alpha)=
		\begin{cases}
			(-1)^{\frac{(n-1)e(\alpha)}{2}}\alpha& n \text{ is odd} \\
			(-1)^{\frac{e(\alpha)}{2}}\alpha& n \text{ is even,}
		\end{cases}
	\end{equation*}
	where $\alpha\in S_n(\M,v)$ is a stated $n$-web in $\MN$. In particular, $\Psi_n$ is an algebra isomorphism when $\MN$ is the thickening of a pb surface and $\Psi_n(\alpha) = \alpha$ for any stated $n$-web $\alpha$ in $\MN$ without endpoints.

\end{theorem}
\begin{proof}
	It is obvious that $\Psi_n$ is well-defined in the set of isotopy classes of stated $n$-webs in $\MN$. Then it suffices to show that $\Psi_n$ preserves relations \eqref{w.cross}-\eqref{wzh.eight}.
	
	Let $u = -v$. From the previous discussion, we know $q_u^{\frac{1}{n}} = q_v^{\frac{1}{n}},q_u = q_v$. Then, for each $1\leq i\leq n$, we have 
	\begin{align}\label{ci}
	c_{i,u} = (-q_u)^{n-i}(q_u^{\frac{1}{2n}})^{n-1} = (-1)^{n-1}(-q_v^{n-i})q_v^{\frac{n-1}{2n}} = (-1)^{n-1}c_{i,v}.
	\end{align}
	
	When $n$ is odd, we suppose $n = 2k+1$. Equation \eqref{ci} implies that
	$c_{i,u} = c_{i,u}$ for $1\leq i\leq n$. We have 
	$q_u^{\frac{1}{2}} = u^n = (-v)^n = -v^n = -q_v^{\frac{1}{2}}$, and 
	$a_u = q_u^{\frac{n+1}{4}}q_u^{-\frac{n^2}{2}} = q_u^{\frac{k+1}{2}}q_u^{\frac{-n^2}{2}} = (-1)^k  q_v^{\frac{k+1}{2}}q_v^{\frac{-n^2}{2}} = (-1)^k a_v.$ Trivially $\Psi_n$ preserves relations \eqref{w.cross}-\eqref{wzh.four} and \eqref{wzh.six}-\eqref{wzh.eight}. Let $\alpha$ (resp. $\alpha_{\sigma}$) be the stated $n$-web on the left-hand (resp. right-hand) side of relation \eqref{wzh.five}, then we have $e(\alpha_{\sigma}) = e(\alpha)+n$. We have 
	\begin{align*}
		\Psi_n(a_v\sum_{\sigma\in \fS}(-q_v)^{\ell(\sigma)}\alpha_{\sigma})
		= (-1)^k a_u\sum_{\sigma\in \fS}(-q_u)^{\ell(\sigma)} (-1)^{k(e(\alpha)+n)} \alpha_{\sigma} = (-1)^{ke(\alpha)}\alpha =\Psi_n(\alpha).
	\end{align*}
	
	When $n$ is even, supppose $n = 2\lambda$. Equation \eqref{ci} implies that
	$c_{i,u} = -c_{i,u}$ for $1\leq i\leq n$. We have 
	$q_u^{\frac{1}{2}} = u^n = (-v)^n = v^n = q_v^{\frac{1}{2}}$,  
	$q_u^{\frac{1}{4}} = u^\lambda = (-1)^\lambda v^\lambda = (-1)^\lambda q_v^{\frac{1}{4}}$, and $a_u = q_u^{\frac{n+1-2n^2}{4}} = (-1)^{\lambda(n+1-2n^2)} q_v^{\frac{n+1-2n^2}{4}} = (-1)^\lambda a_v.$ Trivially $\Psi_n$ preserves relations \eqref{w.cross}-\eqref{wzh.four}. 
	
	Relation \eqref{wzh.five}: We use $\beta_1$ (resp. $\beta_{1,\sigma}$) to denote the stated $n$-web on the left-hand (resp. right-hand) side of relation \eqref{wzh.five}. Then $e(\beta_{1,\sigma}) = e(\beta_{1})+n$. We have 
	\begin{align*}
		\Psi_n(a_v\sum_{\sigma\in \fS}(-q_v)^{\ell(\sigma)} \beta_{1,\sigma} ) &=(-1)^\lambda a_u\sum_{\sigma\in \fS}(-q_u)^{\ell(\sigma)} (-1)^{\frac{1}{2}(e(\beta_1)+n)} \beta_{1,\sigma} \\& = (-1)^{\frac{e(\beta_1)}{2}}\beta_1 = \Psi_n(\beta_1).
	\end{align*}
	
	Relation \eqref{wzh.six}: We have 
	$$\Psi_n(\raisebox{-.20in}{
		\begin{tikzpicture}
			\tikzset{->-/.style=
				{decoration={markings,mark=at position #1 with
						{\arrow{latex}}},postaction={decorate}}}
			\filldraw[draw=white,fill=gray!20] (-0.7,-0.7) rectangle (0,0.7);
			\draw [line width =1.5pt,decoration={markings, mark=at position 1 with {\arrow{>}}},postaction={decorate}](0,0.7)--(0,-0.7);
			\draw [color = black, line width =1pt] (0 ,0.3) arc (90:270:0.5 and 0.3);
			\node [right]  at(0,0.3) {$i$};
			\node [right] at(0,-0.3){$j$};
			\filldraw[fill=white,line width =0.8pt] (-0.5 ,0) circle (0.07);
	\end{tikzpicture}}) = -\raisebox{-.20in}{
	\begin{tikzpicture}
	\tikzset{->-/.style=
		{decoration={markings,mark=at position #1 with
				{\arrow{latex}}},postaction={decorate}}}
	\filldraw[draw=white,fill=gray!20] (-0.7,-0.7) rectangle (0,0.7);
	\draw [line width =1.5pt,decoration={markings, mark=at position 1 with {\arrow{>}}},postaction={decorate}](0,0.7)--(0,-0.7);
	\draw [color = black, line width =1pt] (0 ,0.3) arc (90:270:0.5 and 0.3);
	\node [right]  at(0,0.3) {$i$};
	\node [right] at(0,-0.3){$j$};
	\filldraw[fill=white,line width =0.8pt] (-0.5 ,0) circle (0.07);
\end{tikzpicture}} =-\delta_{\bar j,i }\,  c_{i,u}\ \raisebox{-.20in}{
\begin{tikzpicture}
\tikzset{->-/.style=
	{decoration={markings,mark=at position #1 with
			{\arrow{latex}}},postaction={decorate}}}
\filldraw[draw=white,fill=gray!20] (-0.7,-0.7) rectangle (0,0.7);
\draw [line width =1.5pt](0,0.7)--(0,-0.7);
\end{tikzpicture}} =\delta_{\bar j,i }\,  c_{i,v}\ \raisebox{-.20in}{
\begin{tikzpicture}
\tikzset{->-/.style=
	{decoration={markings,mark=at position #1 with
			{\arrow{latex}}},postaction={decorate}}}
\filldraw[draw=white,fill=gray!20] (-0.7,-0.7) rectangle (0,0.7);
\draw [line width =1.5pt](0,0.7)--(0,-0.7);
\end{tikzpicture}}\,.$$

Relation \eqref{wzh.seven}: We use $\beta_2$ (resp. $\beta_2'$) to denote the stated $n$-web on the left-hand (resp. right-hand) side of relation \eqref{wzh.seven}. Then $e(\beta_{2}') = e(\beta_{2})+2$. We have 
$$\Psi_n(\sum_{1\leq i\leq n}c_{\bar{i},v}^{-1} \beta_2')
= \sum_{1\leq i\leq n}-c_{\bar{i},u}^{-1} (-1)^{\frac{1}{2}(e(\beta_2)+2)} \beta_2'  = (-1)^{\frac{1}{2}e(\beta_2)} \beta_2 = \Psi_n(\beta_2).$$

It is a trivial check that $\Psi_n$ preserves relation \eqref{wzh.eight}.
\end{proof}

\begin{rem}
	Theorem \ref{5.1} is also true when $\MN$ is the generalized marked 3-manifold, please refer to the next section  for the definition of the generalized marked 3-manifold.
\end{rem}

Let \(\MN\) be a marked 3-manifold, and let \( e \) be an oriented open interval in \( \partial M \) such that the closure of \( \mathcal{N} \) is disjoint from the closure of \( e \). Then \((M, \mathcal{N} \cup e)\) is also a marked 3-manifold, and we say that \((M, \mathcal{N} \cup e)\) is obtained from \(\MN\) by adding an extra marking \( e \).  
We denote by \( l_e \) the \( R \)-linear map  
\[
l_e: S_n(\M, v) \to S_n(M, \mathcal{N} \cup e, v)
\]  
induced by the embedding of \(\MN\) into \((M, \mathcal{N} \cup e)\). We refer to \( l_e \) as the {\bf adding marking map}.  

We now state the following immediate Corollary.

\begin{corollary}\label{minus}
	 For each positive integer $n$,  the following  diagram commutes.
	  \begin{equation*}
	 	\begin{tikzcd}
	 		S_n(\M,v)  \arrow[r, "l_e"]
	 		\arrow[d, "\Psi_n"]  
	 		&  S_n(M,\mathcal{N}\cup e,v) \arrow[d, "\Psi_n"] \\
	 		S_n(\M,-v)  \arrow[r, "l_e"] 
	 		&  S_n(M,\mathcal{N}\cup e,-v)\\
	 	\end{tikzcd}.
	 \end{equation*}
\end{corollary}

\def \cN {\mathcal{N}}

\section{Stated $SL_n$-TQFT}\label{secTQFT}
The idea of this section is motivated by Costantino and L{\^e}'s  paper \cite{CL2022TQFT}. They formulated the stated TQFT theory for $SL_2$. In this section, we will generalize their results to $SL_n$.

A {\bf generalized marked 3-manifold} a pair $\MN$, where $M$ is an oriented 3-manifold, and $\mathcal{N}\subset \partial M$ is a one dimensional submanifold consisting of oriented circles and oriented open intervals such that  there is no intersection between the closure of any two components. We use $\cN_1$ to denote the subset of $\cN$ consisting of all the oriented open intervals.

The definition of the stated $SL_n$-skein module of a generalized marked 3-manifold is the same with the definition of the stated $SL_n$-skein module of a marked 3-manifold. For a generalized marked 3-manifold $\MN$, we also use $S_n(\M,v)$ to denote the stated $SL_n$-skein module of $\MN$. The classical limit of the stated $SL_n$-skein module of the generalized marked 3-manifold is well-studied in \cite{wang2023stated}.

A {\bf marked surface} is a pair $(\Sigma,\mathcal{P})$, where $\Sigma$ is an oriented compact surface and $\mathcal{P}$ is a set of finite  points in $\partial\Sigma$, called marked points. We assume that every point of $\mathcal{P}$ is signed by ``$-$" or ``$+$", and every component of $\Sigma$ contains at least one marked point.

\def\MP{(\Sigma,\mathcal{P})}
\def\P{\Sigma,\mathcal{P}}

For a marked surface $\MP$, we can define a marked 3-manifold $\MN$, where $M = \Sigma\times [-1,1], \cN =\cP\times(-1,1)$. For any $p\in\cP$, the orientation of $\{p\}\times (-1,1)$ is the positive (resp. negative) orientation of $(-1,1)$ if the sign of $p$ is positive (resp. negative). We call $\MN$ the thickening of $\MP$, and define $S_n(\P,v)$ to be $S_n(\M,v)$. Then $S_n(\P,v)$ has an algebra structure given by stacking the stated $n$-webs, that is, for any two stated $n$-webs $\alpha,\beta$, the product $\alpha\beta$ is defined to be stacking $\alpha$ above $\beta$.

Let $\MN$ be a generalized marked 3-manifold, and $\MP$ be a marked surface. Suppose $\phi:\Sigma\rightarrow \partial M$ is an embedding such that $\phi(\Sigma)\cap \overline{\cN}=\phi(\Sigma)\cap \overline{\cN_1} = \phi(\cP)$. For any $p\in \cP$, we assume the sign of $p$ is $+$ (resp. $-$) if $\phi$ is orientation preserving and the component of $\overline{\cN}$ connecting to $p$ points towards (resp. away from) $p$, and the sign of $p$ is $-$ (resp. $+$) if $\phi$ is orientation reversing and the component of $\overline{\cN}$ connecting to $p$ points towards (resp. away from) $p$. 

Then a closed regular neighborhood $U(\Sigma)$ of $\phi(\Sigma)$ is isomorphic to the thickening of $\MP$, we use $\psi$ to denote this isomorphism from $\MP\times [-1,1] $ to $U(\Sigma)$. Then $S_n(\P,v)$ has a left (resp. right) action on $S_n(\M,v)$ if $\phi$ is orientation preserving (resp. reversing). For any stated $n$-webs $\alpha$ in $\MP\times [-1,1]$
and $\beta\in\MN$, we first isotope $\beta$ such that $\beta$ is away from 
$U(\Sigma)$, the action of $\alpha$ on $\beta$ is given by $\psi(\alpha)\cup\beta$. Costantino and L{\^e} defined the above actions for $n=2$ in subsection 4.5 in \cite{CL2022TQFT}. Here we recalled their construction and generalized it to all $n$ in an obvious way. 

\def\Sl{{\mathsf{Sl}}}

For a generalized marked  3-manifold $\MN$, Costantino and and L{\^e} defined the so called strict subsurface
in \cite{CL2022TQFT}. Here we recall their definition.
A {\bf strict subsurface} $\Sigma$ of $(M,\cN)$ is a proper  embedding $ \Sigma\rightarrow M$ of a compact surface such that $\Sigma$ is traversal to $\cN$ and  every connected component of $\Sigma$ intersects $\cN$.
Define  $\Sl_{\Sigma}\MN$ to be $(M', \cN')$, where $M' = M \setminus\Sigma$ and $\cN' = \cN\setminus \Sigma$.
Define $\cP = \Sigma\cap\cN$.
For a point $p\in\cP$, define its sign to be $+$ or $-$
according as the orientation of $M$ is equal the orientation of $\Sigma$ followed by the orientation
of the tangent to $\cN$ at $p$ or not. Then $\MP$ is a marked surface and there is a right
and a left action of $S_n(\P,v)$ on $S_n(\Sl_{\Sigma}\MN,v)$.

The obvious embedding from $\Sl_{\Sigma}\MN$ to $\MN$  induces an $R$-linear map $$\varphi_{\Sigma}:S_n(\Sl_{\Sigma}\MN,v)\rightarrow S_n(\M,v).$$

\begin{theorem}\label{TQFT}
	Assume that $\Sigma$ is a strict subsurface of a generalized marked 3-manifold $\MN$. Then the linear map $\varphi_{\Sigma}:S_n(\Sl_{\Sigma}\MN,v)
	\rightarrow S_n(\M,v)$ is surjective and its kernel is the $R$-span of $\{a\cdot x - x\cdot a\mid a\in S_n(\P,v), x\in  S_n(\Sl_{\Sigma}\MN,v)\}$.
\end{theorem}

Note that $$S_n(\Sl_{\Sigma}\MN,v)/R\text{-span}\{a\cdot x - x\cdot a\mid a\in S_n(\P,v), x\in  S_n(\Sl_{\Sigma}\MN,v)\},$$ denoted as $\text{HH}_0(S_n(\Sl_{\Sigma}\MN,v))$, is the 0-th Hochschild homology of the $S_n(\P,v)$-bimodule $S_n(\Sl_{\Sigma}\MN,v)$.

\begin{proof}
The proof of Theorem 5.1 in \cite{CL2022TQFT} applies almost directly in this setting. The only remaining step is to verify the following equations, along with their counterparts obtained by reversing the arrows of all stated \( n \)-webs, in \( \text{HH}_0(S_n(\Sl_{\Sigma} \MN, v)) \):
	\begin{align}\label{eq1}
		\raisebox{-.90in}{
			\begin{tikzpicture}
				\tikzset{->-/.style=
					{decoration={markings,mark=at position #1 with
							{\arrow{latex}}},postaction={decorate}}}
				\filldraw[draw=white,fill=gray!20] (-2.4,0.4) rectangle (0,2.8);
				\draw [line width =1.5pt,decoration={markings, mark=at position 1 with {\arrow{>}}},postaction={decorate}](0,2.8)--(0,0.4);
				\filldraw[draw=white,fill=gray!20] (-2.4,0.2) rectangle (0,-2.2);
				\draw [line width =1.5pt,decoration={markings, mark=at position 1 with {\arrow{>}}},postaction={decorate}](0,0.2)--(0,-2.2);
				\draw [line width =1pt,decoration={markings, mark=at position 1 with {\arrow{>}}},postaction={decorate}](-1.2,2.8)--(-1.2,2.4);
				\draw [line width =1pt,decoration={markings, mark=at position 1 with {\arrow{>}}},postaction={decorate}](-1.6,2.8)--(-1.6,2.4);
				\draw [line width =1pt,decoration={markings, mark=at position 1 with {\arrow{>}}},postaction={decorate}](-0.4,2.8)--(-0.4,2.4);
				\draw [line width =1pt](-1.2,2.4)--(-1.2,1.2);
				\draw [line width =1pt](-1.6,2.4)--(-1.2,1.2);
				\draw [line width =1pt](-0.4,2.4)--(-1.2,1.2);
				\node  at(-0.8,2.2) {$\cdots$};
		\end{tikzpicture}}
		=
		\sum_{1\leq i_1,\cdots,i_n\leq n}c_{\overline{i_1},v}^{-1}\cdots c_{\overline{i_n},v}^{-1}
		\raisebox{-.90in}{
			\begin{tikzpicture}
				\tikzset{->-/.style=
					{decoration={markings,mark=at position #1 with
							{\arrow{latex}}},postaction={decorate}}}
				\filldraw[draw=white,fill=gray!20] (-2.4,0.4) rectangle (0,2.8);
				\draw [line width =1.5pt,decoration={markings, mark=at position 1 with {\arrow{>}}},postaction={decorate}](0,2.8)--(0,0.4);
				\filldraw[draw=white,fill=gray!20] (-2.4,0.2) rectangle (0,-2.2);
				\draw [line width =1.5pt,decoration={markings, mark=at position 1 with {\arrow{>}}},postaction={decorate}](0,0.2)--(0,-2.2);
				\node [right]  at(0,2) {$i_1$};
				\node [right] at(0,1.2){$i_{n-1}$};
				\node [right]  at(0,0.8) {$i_n$};
				\node[left]  at(0,1.7){$\vdots$};
				\draw [line width =1pt,decoration={markings, mark=at position 1 with {\arrow{>}}},postaction={decorate}](-1.2,2.8)--(-1.2,2.4);
				\draw [line width =1pt,decoration={markings, mark=at position 1 with {\arrow{>}}},postaction={decorate}](-1.6,2.8)--(-1.6,2.4);
				\draw [line width =1pt,decoration={markings, mark=at position 1 with {\arrow{>}}},postaction={decorate}](-0.4,2.8)--(-0.4,2.4);
				\draw [line width =1pt] (-0.4 ,2.4) arc (180:270:0.4);
				\draw [line width =1pt] (-1.2 ,2.4) arc (180:270:1.2);
				\draw [line width =1pt] (-1.6 ,2.4) arc (180:270:1.6);
				\draw [line width =1pt,decoration={markings, mark=at position 0.4 with {\arrow{>}}},postaction={decorate}](-1.2,-1.6)--(-1.2,-2.0);
				\draw [line width =1pt,decoration={markings, mark=at position 0.4 with {\arrow{>}}},postaction={decorate}](-1.6,-1.6)--(-1.2,-2.0);
				\draw [line width =1pt,decoration={markings, mark=at position 0.75 with {\arrow{>}}},postaction={decorate}](0,-1.2)--(-1.2,-2.0);
				\draw[line width =1pt] (0,-0.6) parabola bend (-0.4,-0.5) (-1.2,-1.6);
				\draw[line width =1pt] (0,-0.2) parabola bend (-0.6,0) (-1.6,-1.6);
				\node [right]  at(0,-0.2) {$\overline{i_n}$};
				\node [right] at(0,-0.6){$\overline{i_{n-1}}$};
				\node [right]  at(0,-1.2) {$\overline{i_1}$};
				\node[left]  at(0,-0.9){$\vdots$};
		\end{tikzpicture}}
	\end{align}

	\begin{align}\label{eq2}
		\raisebox{-.90in}{
			\begin{tikzpicture}
				\tikzset{->-/.style=
					{decoration={markings,mark=at position #1 with
							{\arrow{latex}}},postaction={decorate}}}
				\filldraw[draw=white,fill=gray!20] (-2.4,0.4) rectangle (0,2.8);
				\draw [line width =1.5pt,decoration={markings, mark=at position 1 with {\arrow{>}}},postaction={decorate}](0,2.8)--(0,0.4);
				\filldraw[draw=white,fill=gray!20] (-2.4,0.2) rectangle (0,-2.2);
				\draw [line width =1.5pt,decoration={markings, mark=at position 1 with {\arrow{>}}},postaction={decorate}](0,0.2)--(0,-2.2);
				\node [right]  at(0,2) {$i$};
				\node [right] at(0,1.2){$j$};
				\draw [line width =1pt,decoration={markings, mark=at position 1 with {\arrow{>}}},postaction={decorate}](-1.2,2.8)--(-1.2,2.4);
				\draw [line width =1pt,decoration={markings, mark=at position 1 with {\arrow{<}}},postaction={decorate}](-0.4,2.8)--(-0.4,2.4);
				\draw [line width =1pt] (-1.2 ,2.4) -- (-1.2 ,1.6) ;
				\draw [line width =1pt] (-0.4 ,2.4) -- (-0.4 ,1.6) ;
				\draw [line width =1pt] (-1.2 ,1.6) arc (180:360:0.4);
		\end{tikzpicture}}
		=
		\sum_{1\leq i,j\leq n}c_{\bar{i},v}^{-1}c_{\bar{j},v}^{-1}
		\raisebox{-.90in}{
			\begin{tikzpicture}
				\tikzset{->-/.style=
					{decoration={markings,mark=at position #1 with
							{\arrow{latex}}},postaction={decorate}}}
				\filldraw[draw=white,fill=gray!20] (-2.4,0.4) rectangle (0,2.8);
				\draw [line width =1.5pt,decoration={markings, mark=at position 1 with {\arrow{>}}},postaction={decorate}](0,2.8)--(0,0.4);
				\filldraw[draw=white,fill=gray!20] (-2.4,0.2) rectangle (0,-2.2);
				\draw [line width =1.5pt,decoration={markings, mark=at position 1 with {\arrow{>}}},postaction={decorate}](0,0.2)--(0,-2.2);
				\node [right]  at(0,2) {$i$};
				\node [right] at(0,1.2){$j$};
				\draw [line width =1pt,decoration={markings, mark=at position 1 with {\arrow{>}}},postaction={decorate}](-1.2,2.8)--(-1.2,2.4);
				\draw [line width =1pt,decoration={markings, mark=at position 1 with {\arrow{<}}},postaction={decorate}](-0.4,2.8)--(-0.4,2.4);
				\draw [line width =1pt] (-0.4 ,2.4) arc (180:270:0.4);
				\draw [line width =1pt] (-1.2 ,2.4) arc (180:270:1.2);
				\draw [line width =1pt,decoration={markings, mark=at position 1 with {\arrow{>}}},postaction={decorate}](0 ,-0.4) -- (-0.3 ,-0.4);
				\draw [line width =1pt,decoration={markings, mark=at position 1 with {\arrow{<}}},postaction={decorate}](-0 ,-1.4) -- (-0.3 ,-1.4);
				\draw [line width =1pt] (-0.3 ,-0.4) arc (90:270:0.5);
				\node [right]  at(0,-0.4) {$\bar{j}$};
				\node [right] at(0,-1.4){$\bar{i}$};
		\end{tikzpicture}}
	\end{align}

	\begin{align}\label{eq3}
		\raisebox{-.90in}{
			\begin{tikzpicture}
				\tikzset{->-/.style=
					{decoration={markings,mark=at position #1 with
							{\arrow{latex}}},postaction={decorate}}}
				\filldraw[draw=white,fill=gray!20] (-2.4,0.4) rectangle (0,2.8);
				\draw [line width =1.5pt,decoration={markings, mark=at position 1 with {\arrow{>}}},postaction={decorate}](0,2.8)--(0,0.4);
				\filldraw[draw=white,fill=gray!20] (-2.4,0.2) rectangle (0,-2.2);
				\draw [line width =1.5pt,decoration={markings, mark=at position 1 with {\arrow{>}}},postaction={decorate}](0,0.2)--(0,-2.2);
				\node [right] at(0,1.2){$j$};
				\draw [line width =1pt,decoration={markings, mark=at position 1 with {\arrow{>}}},postaction={decorate}](-1.2,2.8)--(-1.2,2.4);
				\draw [line width =1pt] (-1.2 ,2.4) arc (180:270:1.2);
		\end{tikzpicture}}
		=
		\sum_{1\leq i\leq n}c_{\bar{i},v}^{-1}
		\raisebox{-.90in}{
			\begin{tikzpicture}
				\tikzset{->-/.style=
					{decoration={markings,mark=at position #1 with
							{\arrow{latex}}},postaction={decorate}}}
				\filldraw[draw=white,fill=gray!20] (-2.4,0.4) rectangle (0,2.8);
				\draw [line width =1.5pt,decoration={markings, mark=at position 1 with {\arrow{>}}},postaction={decorate}](0,2.8)--(0,0.4);
				\filldraw[draw=white,fill=gray!20] (-2.4,0.2) rectangle (0,-2.2);
				\draw [line width =1.5pt,decoration={markings, mark=at position 1 with {\arrow{>}}},postaction={decorate}](0,0.2)--(0,-2.2);
				\node [right] at(0,1.2){$i$};
				\draw [line width =1pt,decoration={markings, mark=at position 1 with {\arrow{>}}},postaction={decorate}](-1.2,2.8)--(-1.2,2.4);
				\draw [line width =1pt] (-1.2 ,2.4) arc (180:270:1.2);
				\draw [line width =1pt,decoration={markings, mark=at position 1 with {\arrow{>}}},postaction={decorate}](0 ,-0.4) -- (-0.3 ,-0.4);
				\draw [line width =1pt,decoration={markings, mark=at position 1 with {\arrow{<}}},postaction={decorate}](-0 ,-1.4) -- (-0.3 ,-1.4);
				\draw [line width =1pt] (-0.3 ,-0.4) arc (90:270:0.5);
				\node [right]  at(0,-0.4) {$\bar{i}$};
				\node [right] at(0,-1.4){$j$};
		\end{tikzpicture}}
	\end{align}

Equation \eqref{eq1} and its counterpart follow from equation (53) in \cite{le2021stated} and relation \eqref{wzh.five}.  
Equation \eqref{eq2} and its counterpart arise from relations \eqref{wzh.six} and \eqref{wzh.seven}.  
Equation \eqref{eq3} and its counterpart are derived from relation \eqref{wzh.six}.  

By applying equations \eqref{eq1}–\eqref{eq3} along with the techniques used in the proof of Theorem 5.1 in \cite{CL2022TQFT}, we can easily establish Theorem \ref{TQFT}.

\end{proof}

A {\bf decorated manifold} is 5-tuple $\mathbb{M}= (M,\partial^{+}M,\partial^{-}M,\partial^{s}M,\mathcal{N})$ with the following properties: (1) $\MN$ is a generalized marked  3-manifold; (2) $\partial^{+}M, \partial^{-}M, \partial^{s}M\subset \partial M$ are compact surfaces with boundaries with disjoint interiors and
oriented as induced by the orientation of $M$, such that $\partial^{+}M\cup\partial^{-}M\cup \partial^{s}M= \partial M$ and $\partial^{+}M\cap\partial^{-}M=\emptyset$; (3) $(\partial^{\pm1}M,\overline{\mathcal{N}}\cap \partial^{\pm1}M)$ are marked surfaces, where the signs of the marked points are determined by the orientation of $\mathcal{N}$ and $\partial^{\pm}M$. For a detailed definition of decorated manifolds, see Definition 6.1 in \cite{CL2022TQFT}.

In Definition 6.2 of \cite{CL2022TQFT}, Costantino and L{\^e} introduced the {\bf category of decorated cobordisms}, denoted by \(\text{DeCob}\). Here we briefly recall its definition.  
The objects of \(\text{DeCob}\) are marked surfaces. A morphism from \(\Sigma_-\) to \(\Sigma_+\) is given by the diffeomorphism class of a decorated manifold  
\[
\mathbb{M} = (M, \partial^+ M, \partial^- M, \partial^s M, \mathcal{N})
\]  
endowed with diffeomorphisms  
\[
\phi_{\pm} : \partial^{\pm} M \to \Sigma_{\pm}
\]  
such that \(\phi_+\) is orientation-preserving and \(\phi_-\) is orientation-reversing. Morphisms are composed by gluing decorated manifolds along a common marked surface. The identity morphism is given by the thickening of the corresponding marked surface. The category \(\text{DeCob}\) is symmetric monoidal, with the monoidal operation \(\otimes\) given by disjoint union.  

Let \(\text{Mor}\) be the {\bf Morita category}, whose objects are \( R \)-algebras and whose morphisms are isomorphism classes of bimodules in the category of \( R \)-modules. The composition of morphisms is given by the tensor product over the intermediate algebra. The identity morphism for an algebra \( A \) is the isomorphism class of \( A \) as a bimodule over itself via left and right multiplication. The category \(\text{Mor}\) is symmetric monoidal, with the monoidal operation given by the tensor product \(\otimes_R\).  

For any object \( \MP \in \text{DeCob} \), we define \( S_n(\MP) \) as \( S_n(\P, v) \), which is an object in \(\text{Mor}\). For any morphism  
\[
\mathbb{M} = (M, \partial^+ M, \partial^- M, \partial^s M, \mathcal{N}): \Sigma_- \to \Sigma_+,
\]  
we define \( S_n(\mathbb{M}) \) as \( S_n(\M, v) \), which is a right module over \( S_n(\Sigma_-) \) and a left module over \( S_n(\Sigma_+) \). Thus, \( S_n(\mathbb{M}) \) defines a morphism from \( S_n(\Sigma_-) \) to \( S_n(\Sigma_+) \) in \(\text{Mor}\).  

The following theorem is the main result of this section. It generalizes Theorem 6.5 of Costantino and L{\^e} in \cite{CL2022TQFT}, where they proved the statement for \( n = 2 \).

\begin{theorem}\label{main}
	For each positive integer $n$, we have that
	$S_n:\text{DeCob}\rightarrow \text{Mor}$ is a symmetric monoidal functor.
\end{theorem}
\begin{proof}
	Costantino and L{\^e}'s proof for Theorem 6.5 in \cite{CL2022TQFT} works here. We briefly recall their proof. 
	
	It is obvious that $S_n$ preserves the symmetric monoidal structure and maps the identity morphism to the identity morphism.  
	It suffices to show that $S_n$ is compatible with the composition of the morphisms. This follows from Theorem \ref{TQFT}.
\end{proof}

\begin{rem}
For an oriented 3-manifold $M$, we denote by $\widetilde{M}$ the marked 3-manifold obtained by removing an open 3-dimensional ball from $M$ and placing a marking on the newly created spherical boundary component.  
For any two oriented 3-manifolds $M_1$ and $M_2$, Theorem 6.9 in \cite{CL2022TQFT} states that  
\begin{align}\label{eq-foundamental}
   S_n(\widetilde{M_1\# M_2}) = S_n(\widetilde{M_1}) \otimes_R S_n(\widetilde{M_2})
\end{align}  
when $n = 2$.

Let \( H_g^+ \) (resp. \( H_g^- \)) be the decorated manifold whose underlying 3-manifold is a handlebody of genus \( g \), with  
\[
\partial^- H_g^+ = (D^2, +) \quad \text{(resp. } \partial^+ H_g^- = (D^2, +) \text{)},
\]  
where \( D^2 \) is a disk with a single marked point. The surface \( \partial^s H_g^\pm \) is a regular neighborhood of \( \partial D^2 \), and  
\[
\partial^+ H_g^+ = \partial H_g^+ \setminus (\partial^- H_g^+ \cup \partial^s H_g^+)  
\quad \text{(resp. } \partial^- H_g^- = \partial H_g^- \setminus (\partial^+ H_g^- \cup \partial^s H_g^-) \text{)}.
\]  
Let \( M = H_g \cup H_g' \) be a Heegaard decomposition of a closed, oriented 3-manifold.
Theorem 6.10 in \cite{CL2022TQFT} states that 
\begin{align}\label{eq-Heegaard}
S_n(\widetilde{M})=S_n(H_g^+)\otimes_{S_n(\partial^+ H_g)}
   S_n((H_g')^{-})
\end{align}
when $n=2$.

Using Theorem \ref{main}, we can easily verify that equations \eqref{eq-foundamental} and \eqref{eq-Heegaard} hold for any general $n$.

\end{rem}

\section{Stated $SL_n$-skein modules at 4-th roots of unity}
In this section, we focus on the stated \( SL_n \)-skein modules at \( \epsilon \in R \) satisfying \( \epsilon^4 = 1 \). For example, when \( R = \mathbb{C} \), we have \( \epsilon = \pm 1, \pm i \), where \( i \in \mathbb{C} \) with \( i^2 = -1 \).  
Let \( \MN \) be a connected marked 3-manifold with \( \mathcal{N} \neq \emptyset \), and let \( (M, \mathcal{N}') \) be a marked 3-manifold obtained from \( \MN \) by adding one extra marking. We will prove that  
\[
S_n(M, \mathcal{N}', \epsilon) = S_n(\M, \epsilon) \otimes_R O_{q_{\epsilon}}(SL_n),
\]  
and that the \( R \)-linear map from \( S_n(\M, \epsilon) \) to \( S_n(M, \mathcal{N}', \epsilon) \), induced by the obvious embedding from \( \MN \) to \( (M, \mathcal{N}') \), is injective.

\def \cN {\mathcal{N}}
\def \SM{S_n(\M,\epsilon)}

In this section, we will always assume $\epsilon\in R$ with $\epsilon^4 = 1$. 
Then $q_{\epsilon}^{\frac{1}{n}} =\pm 1, q_{\epsilon} =\pm 1$.

Let $\MN$ be a  marked  3-manifold. In $S_n(\M,\epsilon)$, we have
\begin{align}\label{root}
\raisebox{-.20in}{	
	\begin{tikzpicture}
		\tikzset{->-/.style=
			{decoration={markings,mark=at position #1 with
					{\arrow{latex}}},postaction={decorate}}}
		\filldraw[draw=white,fill=gray!20] (-0,-0.2) rectangle (1, 1.2);
		\draw [line width =1pt,decoration={markings, mark=at position 0.5 with {\arrow{>}}},postaction={decorate}](0.6,0.6)--(1,1);
		\draw [line width =1pt,decoration={markings, mark=at position 0.5 with {\arrow{>}}},postaction={decorate}](0.6,0.4)--(1,0);
		\draw[line width =1pt] (0,0)--(0.4,0.4);
		\draw[line width =1pt] (0,1)--(0.4,0.6);
		\draw[line width =1pt] (0.4,0.6)--(0.6,0.4);
	\end{tikzpicture}
}=
\raisebox{-.20in}{
	\begin{tikzpicture}
		\tikzset{->-/.style=
			{decoration={markings,mark=at position #1 with
					{\arrow{latex}}},postaction={decorate}}}
		\filldraw[draw=white,fill=gray!20] (-0,-0.2) rectangle (1, 1.2);
		\draw [line width =1pt,decoration={markings, mark=at position 0.5 with {\arrow{>}}},postaction={decorate}](0.6,0.6)--(1,1);
		\draw [line width =1pt,decoration={markings, mark=at position 0.5 with {\arrow{>}}},postaction={decorate}](0.6,0.4)--(1,0);
		\draw[line width =1pt] (0,0)--(0.4,0.4);
		\draw[line width =1pt] (0,1)--(0.4,0.6);
		\draw[line width =1pt] (0.6,0.6)--(0.4,0.4);
	\end{tikzpicture}
},\;
\raisebox{-.20in}{	
	\begin{tikzpicture}
		\tikzset{->-/.style=			
			{decoration={markings,mark=at position #1 with					
					{\arrow{latex}}},postaction={decorate}}}
		\filldraw[draw=white,fill=gray!20] (-0,-0.2) rectangle (1, 1.2);
		\draw [line width =1.5pt,decoration={markings, mark=at position 1 with {\arrow{>}}},postaction={decorate}](1,1.2)--(1,-0.2);
		\draw [line width =1pt](0.6,0.6)--(1,1);
		\draw [line width =1pt](0.6,0.4)--(1,0);
		\draw[line width =1pt] (0,0)--(0.4,0.4);
		\draw[line width =1pt] (0,1)--(0.4,0.6);
		\draw[line width =1pt] (0.4,0.6)--(0.6,0.4);
		\filldraw[fill=white,line width =0.8pt] (0.2 ,0.2) circle (0.07);
		\filldraw[fill=white,line width =0.8pt] (0.2 ,0.8) circle (0.07);
		\node [right]  at(1,1) {$i$};
		\node [right]  at(1,0) {$j$};
	\end{tikzpicture}
} =q_{\epsilon}^{-\frac{1}{n}+\delta_{i,j}}\raisebox{-.20in}{
	\begin{tikzpicture}
		\tikzset{->-/.style=			
			{decoration={markings,mark=at position #1 with					
					{\arrow{latex}}},postaction={decorate}}}
		\filldraw[draw=white,fill=gray!20] (-0,-0.2) rectangle (1, 1.2);
		\draw [line width =1.5pt,decoration={markings, mark=at position 1 with {\arrow{>}}},postaction={decorate}](1,1.2)--(1,-0.2);
		\draw [line width =1pt](0,0.8)--(1,0.8);
		\draw [line width =1pt](0,0.2)--(1,0.2);
		\filldraw[fill=white,line width =0.8pt] (0.2 ,0.8) circle (0.07);
		\filldraw[fill=white,line width =0.8pt] (0.2 ,0.2) circle (0.07);
		\node [right]  at(1,0.8) {$j$};
		\node [right]  at(1,0.2) {$i$};
	\end{tikzpicture}
},\;
\raisebox{-.20in}{	
	\begin{tikzpicture}
		\tikzset{->-/.style=			
			{decoration={markings,mark=at position #1 with					
					{\arrow{latex}}},postaction={decorate}}}
		\filldraw[draw=white,fill=gray!20] (-0,-0.2) rectangle (1, 1.2);
		\draw [line width =1.5pt,decoration={markings, mark=at position 1 with {\arrow{>}}},postaction={decorate}](1,1.2)--(1,-0.2);
		\draw [line width =1pt](0.6,0.6)--(1,1);
		\draw [line width =1pt](0.6,0.4)--(1,0);
		\draw[line width =1pt] (0,0)--(0.4,0.4);
		\draw[line width =1pt] (0,1)--(0.4,0.6);
		\draw[line width =1pt] (0.4,0.6)--(0.6,0.4);
		\filldraw[fill=black,line width =0.8pt] (0.2 ,0.2) circle (0.07);
		\filldraw[fill=white,line width =0.8pt] (0.2 ,0.8) circle (0.07);
		\node [right]  at(1,1) {$i$};
		\node [right]  at(1,0) {$j$};
	\end{tikzpicture}
} =q_{\epsilon}^{\frac{1}{n}-\delta_{i,\bar{j}}}\raisebox{-.20in}{
	\begin{tikzpicture}
		\tikzset{->-/.style=			
			{decoration={markings,mark=at position #1 with					
					{\arrow{latex}}},postaction={decorate}}}
		\filldraw[draw=white,fill=gray!20] (-0,-0.2) rectangle (1, 1.2);
		\draw [line width =1.5pt,decoration={markings, mark=at position 1 with {\arrow{>}}},postaction={decorate}](1,1.2)--(1,-0.2);
		\draw [line width =1pt](0,0.8)--(1,0.8);
		\draw [line width =1pt](0,0.2)--(1,0.2);
		\filldraw[fill=white,line width =0.8pt] (0.2 ,0.8) circle (0.07);
		\filldraw[fill=black,line width =0.8pt] (0.2 ,0.2) circle (0.07);
		\node [right]  at(1,0.8) {$j$};
		\node [right]  at(1,0.2) {$i$};
	\end{tikzpicture}
}\;.
\end{align}

From left to  right in equation \eqref{root}, the first equality comes from relation \eqref{w.cross}, the second equality comes from relation \eqref{wzh.eight}, and the third equality comes from Proposition 3.2 in \cite{wang2023stated}.
Note that $q_{\epsilon}^{-\frac{1}{n}+\delta_{i,j}},\, q_{\epsilon}^{\frac{1}{n}-\delta_{i,\bar{j}}} =\pm 1$. 
If we set $a_{i,j} = q_{\epsilon}^{-\frac{1}{n}+\delta_{i,j}},\, b_{i,j} = q_{\epsilon}^{\frac{1}{n}-\delta_{i,\bar{j}}}$, then we have
\begin{align}\label{very}
a_{i,j}b_{i,\bar{j}} = 1.
\end{align} 

Let \( \MN \) be a connected marked 3-manifold. The following theorem shows that the \( R \)-linear structure of \( S_n(\M, \epsilon) \) depends only on \( M \) and the number of components of \( \mathcal{N} \), but is independent of the specific locations of these components in \( \partial M \).

Recall that
$D$ denotes the  2-dimensional closed disk.

\begin{theorem}\label{t6.1}
	Let $\MN$ be a connected marked 3-manifold. Assume $(M,\cN_1)$ (resp. $(M,\mathcal{N}_2)$) is obtained from $\MN$ by adding one extra marking $e_1$ (resp. $e_2$), and $l_1: S_n(M,\cN,\epsilon)\rightarrow S_n(M,\cN_1,\epsilon)$ (resp. $l_2: S_n(M,\cN,\epsilon)\rightarrow S_n(M,\cN_2,\epsilon)$) is the $R$-linear map induced by the embedding from $\MN$ to $(M,\cN_1)$ (resp. $(M,\cN_2)$). Then there exists an $R$-linear isomorphism $F: S_n(M,\cN_1,\epsilon)\rightarrow S_n(M,\cN_2,\epsilon)$ such that $F\circ l_1 =l_2$.
\end{theorem}
\begin{proof}
 For each $i=1,2$, the oriented closed interval $\overline{e_i}$ is an embedding from $[0,1]$ to $\partial M$. Let $c:[0,1]\rightarrow M$ be a proper smooth embedding such that $c(0) =\overline{e_1}(1)$ and $c(1) = \overline{e_2}(0)$. Let $U(c)$ be a closed regular neighborhood of $c$, that is $U(c)$ is diffeomorphic to $c\times D$, such that $U(c)\cap \partial M = D\times \{c(0),c(1)\}$. We also require that $U(c)\cap \overline{\cN_1} = U(c)\cap\overline{e_1}\neq \overline{e_1}$ and $U(c)\cap \overline{\cN_2} = U(c)\cap\overline{e_2}\neq \overline{e_2}$ are  closed intervals. Then $c\times \partial D$ is a strict subsurface for both $(M,\cN_1)$ and $(M,\cN_2).$
 
  For each $i=1,2$,
 define $e_i' = e_i\setminus U(c)$, and the orientation of $e_i'$ is inherited from $e_i$.
 Let $M' = M\setminus (c\times \text{int}(D))$, $\cN_1' = \cN_1\setminus U(c) = (\cN_1\setminus e_1)\cup e_1'$,  $\cN_2' = \cN_2\setminus U(c)=(\cN_2\setminus e_2)\cup e_2'$, and $h_1:S_n(M',\cN_1',\epsilon)\rightarrow S_n(M,\cN_1,\epsilon)$ (resp. $h_2:S_n(M',\cN_2',\epsilon)\rightarrow S_n(M,\cN_2,\epsilon)$) be the $R$-linear map induced by the embedding from 
 $(M',\cN_1')$ to $(M,\cN_1)$ (resp. from $(M',\cN_2')$ to $(M,\cN_2)$). 
 Theorem \ref{TQFT} implies that $h_1$ (resp. $h_2$) is surjective and $\text{Ker}(h_1)$ (resp. $\text{Ker}(h_2)$) is generated by the relation depicted in the left (resp. right) picture in Figure \ref{relation}.

 
 \begin{figure}[h]  
 	\centering\includegraphics[width=14cm]{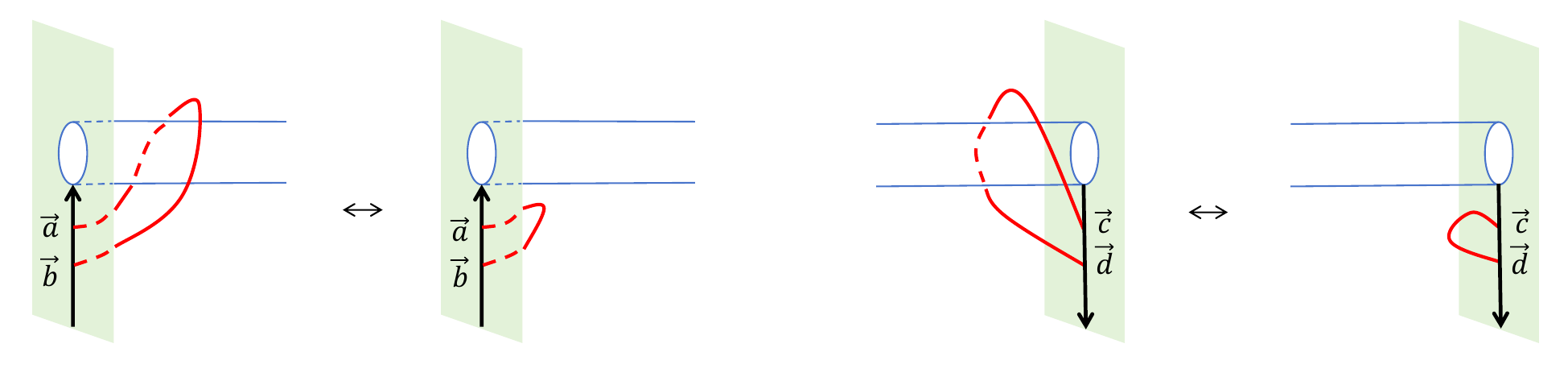} 
 	\caption{The faint green portions belong to $\partial M'$, and the blue sections pertain to $c\times \partial D$. The red curves have arbitrary orientations  and represent parallel copies of stated arcs. The states for these parallel copies of arcs are illustrated by vectors $\vec{a},\vec{b},\vec{c},\vec{d}$ whose entries are integers between $1$ and $n$. In the same relation, the left-hand side and the right-hand side are compatible with each other regarding the orientations  of the red curves.
 		 The black arrow in the left (resp. right) picture  is a part of $e_1'$ (resp. $e_2'$). 		 
 		 The left (resp. right) picture is intended for relations that generate $\text{Ker}(h_1)$ (resp. $\text{Ker}(h_2)$).}       
 	\label{relation}   
 \end{figure}
 
 There is a copy of $c$ in $c\times \partial D$, denoted as $c'$, such that 
 $c'(0) = \overline{e_1'}(1)$ and $c'(1) = \overline{e_2'}(0)$. Let $V(c)\in M'$ be a small open regular neighborhood of $e_1'\cup e_2'\cup c'$ (here the ambient 3-manifold is $M'$) such that $V(c)$ retracts to $e_1'\cup e_2'\cup c'$ and $V(c)\cap \overline{ \cN} =\emptyset$. Let $f$ be the isomorphism from $ (M',\cN_1')$ to $ (M',\cN_2')$ that drags $e_1'$ to $e_2'$ along $c'$ such that $f$ is identity on $M'\setminus V(c)$. The induced $R$-linear map $f_* :S_n(M',\cN_1',\epsilon)\rightarrow S_n (M',\cN_2',\epsilon)$ is illustrated in Figure \ref{map}.
 
  \begin{figure}[h]  
 	\centering\includegraphics[width=14cm]{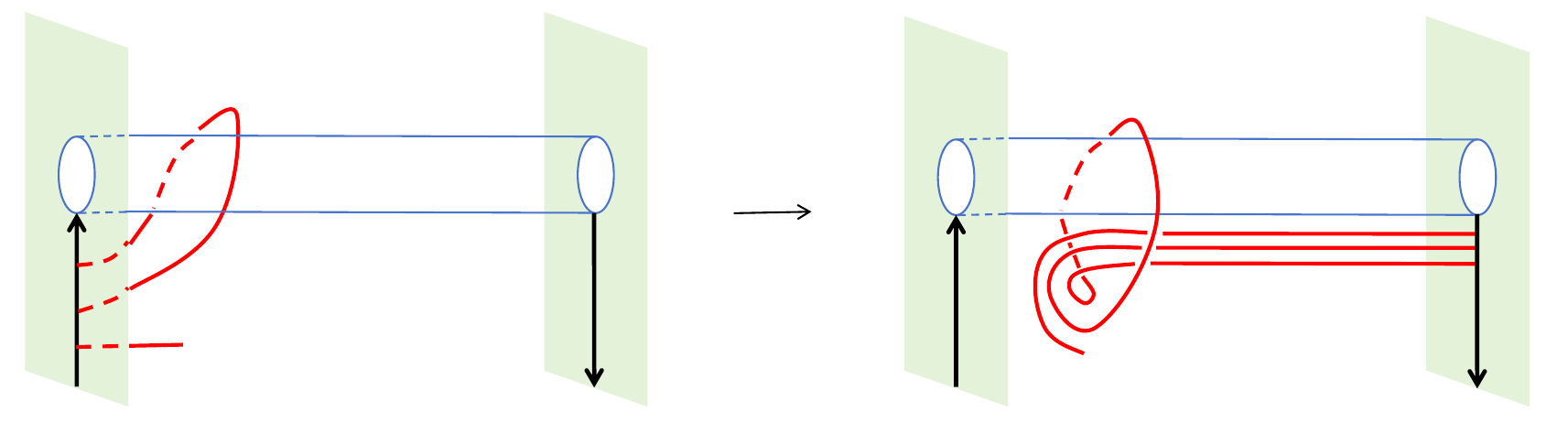} 
 	\caption{The picture illustrates the $R$-linear map $f_*$.}       
 	\label{map}   
 \end{figure}
 
 Using relations in \eqref{root} and equation \eqref{very}, we can easily show $f_{*}(\text{Ker}(h_1)) = \text{Ker}(h_2)$. Similarly, we have $(f^{-1})_{*}(\text{Ker}(h_2)) = \text{Ker}(h_1)$. Then $f_{*}$ (resp. $(f^{-1})_*$) induces an $R$-linear map $F:S_n(M,\cN_1,\epsilon)\rightarrow S_n(M,\cN_2,\epsilon)$ (resp. $G:S_n(M,\cN_2,\epsilon)\rightarrow S_n(M,\cN_1,\epsilon)$). Obviously, the $R$-linear maps $F$ and $G$ are inverse to each other.
 
 From the definition of $F$, it is easy to show $F\circ l_1 = l_2$.
 
\end{proof}

\begin{rem}
	The author proved Theorem \ref{t6.1} for  general quantum parameter (that is, $\epsilon$ is replaced with general invertible  element $v\in R$) when $e_1,e_2$ lie in the same  component of $\partial M$, please refer to Lemma 6.4 in \cite{wang2023stated}.
	
	Note that Theorem \ref{t6.1} is not ture for  general quantum parameter when $e_1,e_2$ lie in different  components of $\partial M$, please refer to \cite{CL2022TQFT} for counterexamples.
\end{rem}

The following Theorem is the main result of this section.

\begin{theorem}\label{inj}
  Let \( \MN \) be a connected marked 3-manifold with \( \mathcal{N} \neq \emptyset \), and let \( (M, \mathcal{N}') \) be another marked 3-manifold obtained from \( \MN \) by adding an extra marking \( e' \).  
Let \( l_{e'} \) denote the \( R \)-linear map  
  \[
  l_{e'}: S_n(\M, \epsilon) \to S_n(M, \mathcal{N}', \epsilon)
  \]  
  induced by the embedding from \( \MN \) into \( (M, \mathcal{N}') \). Then:  

  \begin{enumerate}
      \item The \( R \)-linear map \( l_{e'} \) is injective.  
      \item There is an isomorphism  
      \[
      S_n(M, \mathcal{N}', \epsilon) \simeq S_n(\M, \epsilon) \otimes_R O_{q_{\epsilon}}(SL_n).
      \]
  \end{enumerate}
\end{theorem}
\begin{proof}
	Since $\cN$ is not empty, we suppose $e$ is a component of $\cN$.
	
	Case 1: $e$ and $e'$ lie in a same component of $\partial M$. Theorem 6.10 in \cite{wang2023stated} implies this case. 
	
	Case 2: $e$ and $e'$  lie in different components of $\partial M$. Assume that $(M,\cN'')$ is obtained from $\MN$ by adding one extra marking $e''$ with $e''$ and $e$ lying in a same component of $\partial M$. We use $l_{e''}$ to denote the $R$-linear map from $S_n(\M,\epsilon)\rightarrow S_n(M,\cN'',\epsilon)$ induced by the embedding from $\MN$ to $(M,\cN'')$. From Case 1, we know that $l_{e''}$ is injective and $S_n(M,\cN'',\epsilon)\simeq S_n(\M,\epsilon)\otimes_{R} O_{q_{\epsilon}}(SL_n)$. Then Theorem \ref{t6.1} implies $\text{Ker}(l_{e'})=\text{Ker}(l_{e''}) = 0$ and $S_n(M,\cN',\epsilon)\simeq S_n(M,\cN'',\epsilon)\simeq S_n(\M,\epsilon)\otimes_{R} O_{q_{\epsilon}}(SL_n)$. 
\end{proof}

\begin{rem}
	The author proved Theorem \ref{inj} for $\epsilon=1$ using a different and more complicated technique in \cite{wang2023stated}.  
\end{rem}

\section{Applications: Injectivity of the splitting map in special cases}


L{\^e} and Sikora formuated the following Conjecture regarding the injectivity of the splitting map.

\begin{conjecture}[Conjecture 7.12 in \cite{le2021stated}]\label{conj}
	For any pb surface $\Sigma$ and any ideal arc $c$, the
	splitting map $\Theta_c:S_n(\Sigma,v)\rightarrow S_n(\text{Cut}_c(\Sigma),v)$ is injective.
\end{conjecture}

In this section, we recall and utilize some results from \cite{wang2023stated}. Although the author worked with \( \mathbb{C} \) in \cite{wang2023stated} instead of a general commutative domain, it is straightforward to verify that all referenced results remain valid for a general commutative domain \( R \), as the proving techniques in \cite{wang2023stated} apply in this broader setting.  

L{\^e} proved Conjecture \ref{conj} for \( n=2 \) \cite{le2018triangular}, while Higgins established the case \( n=3 \) \cite{higgins2020triangular}. L{\^e} and Sikora further proved Conjecture \ref{conj} when \( \Sigma \) is connected and has a non-empty boundary. The author demonstrated that the splitting map is injective for all marked 3-manifolds when the quantum parameter \( v = 1 \) \cite{wang2023stated}.  

Let \( \MN \) be a marked 3-manifold, let \( D \) be a properly embedded disk in \( M \), and let \( e \) be an open oriented interval in \( D \). The author proved that \( \Theta_{(D,e)} \) is injective if the component of \( \partial M \) containing \( \partial D \) has at least one marking (Corollary 6.11 in \cite{wang2023stated}).  

Using the results from the previous sections, we will prove Conjecture \ref{conj} for \( v^{2m} = 1 \) when \( m \) divides \( n \). Additionally, we will establish that any splitting map for the marked 3-manifold \( \MN \) is injective when every component of \( M \) contains at least one marking and \( v^4 = 1 \). Finally, we will show that the splitting map is injective for any marked 3-manifold and any splitting disk when \( v = -1 \).

We call the pb surface, obtained from $D$ by removing one ideal point in $\partial D$, as {\bf monogon}, denoted as $D_1$. Then it is easy to show $S_n(D_1,v) \simeq R$ as $R$-algebras.

Let $\Sigma$ be a pb surface, and $p$ be an ideal point of $\Sigma$. Suppose $c_p$ is a trivial ideal arc at $p$, that is, the two endpoints of $c_p$ are both $p$ and $c_p$ bounds an embedded monogon. Then $\text{Cut}_{c_p}(\Sigma)$ is the disjoint union of a monogon and $\Sigma_p$.  We use $\Theta_p$ to denote the following composition:
$$S_n(\Sigma,v)\rightarrow S_n(\Sigma_p,v)\otimes_{R} S_n(D_1,v)\simeq S_n(\Sigma_p,v),$$
where the map from $S_n(\Sigma,v)$ to $S_n(\Sigma_p,v)\otimes_{R} S_n(D_1,v)$ is the splitting map $\Theta_{c_p}$. For any stated $n$-web diagram $\alpha$ in $\Sigma$, we can isotope $\alpha$ such that there is no intersection between $\alpha$ and the monogon bounded by $c_p$. Then we have $\Theta_p(\alpha) = \alpha$. Actually this is another way to define $\Theta_p$.

\begin{lemma}[Corollary 8.2 in \cite{le2021stated}]\label{lll}
	Let $\Sigma$ be a pb surface with an ideal point $p$ and an  ideal arc $c$. We have $\text{Ker}(\Theta_{p}) = \text{Ker}(\Theta_c)$.
\end{lemma}

\begin{rem}
	If we analyze \( \Theta_p \) in the thickening of pb surfaces, it becomes evident that \( \Theta_p \) is, in fact, the adding marking map. Consequently, Lemma \ref{lll} asserts that the kernel of the adding marking map coincides with the kernel of the splitting map. The author established this result for general marked 3-manifolds. Specifically, the kernel of the adding marking map equals the kernel of the splitting map when the added marking and the boundary of the splitting disk belong to the same component of \( \partial M \) (which is always the case for the thickening of connected pb surfaces), see Theorem 6.7 in \cite{wang2023stated}.  

L{\^e} and Sikora further proved that \( \text{Ker}(\Theta_p) = \text{Ker}(\Theta_{p'}) \) for any two ideal points \( p, p' \) of a connected pb surface (Theorem 8.1 in \cite{le2021stated}). If we interpret \( \Theta_p \) and \( \Theta_{p'} \) as adding marking maps, the reason behind the equality \( \text{Ker}(\Theta_p) = \text{Ker}(\Theta_{p'}) \) becomes particularly transparent; see Lemma 6.4 in \cite{wang2023stated}. In that lemma, the author proved that the kernels of two adding marking maps are identical if the two added markings lie on the same boundary component of the marked 3-manifold.

\end{rem}

\begin{lemma}\label{key}
	Let $\Sigma$ be a pb surface with an ideal point $p$, let $m,n$ be two positive integers with $m\mid n$, and let $\epsilon\in R$ with $\epsilon^{2m} = 1$. Then $\Theta_{p}:S_n(\Sigma,\epsilon)\rightarrow S_n(\Sigma_p,\epsilon)$ is injective.
\end{lemma}
\begin{proof}
	From Theorem \ref{c_iso}, we know that there exist  linear isomorphisms
	$\varphi_{\epsilon}:S_n(\Sigma,1)\rightarrow S_n(\Sigma,\epsilon)$ and
	$\varphi_{\epsilon}:S_n(\Sigma_p,1)\rightarrow S_n(\Sigma_p,\epsilon)$. The definition of $\varphi_{\epsilon}$ shows that the following  diagram commutes: 
	\begin{equation}\label{com}
		\begin{tikzcd}
			S_n(\Sigma,1)  \arrow[r, "\varphi_{\epsilon}"]
			\arrow[d, "\Theta_p"]  
			&  S_n(\Sigma,\epsilon) \arrow[d, "\Theta_p"] \\
			S_n(\Sigma_p,1)  \arrow[r, "\varphi_{\epsilon}"] 
			&  S_n(\Sigma_p,\epsilon)\\
		\end{tikzcd}.
	\end{equation}
	Corollary 6.8 in \cite{wang2023stated} and Lemma \ref{lll} imply that $\Theta_p:S_n(\Sigma,1)\rightarrow S_n(\Sigma_p,1)$ is injective. 
	Then $\Theta_p:S_n(\Sigma,\epsilon)\rightarrow S_n(\Sigma_p,\epsilon)$ is also injective
	because the two $\varphi_{\epsilon}$ in diagram \eqref{com} are $R$-linear isomorphisms. 
\end{proof}

The following Theorem provides affirmative examples for Conjecture \ref{conj}.

\begin{theorem}
		Let $\Sigma$ be a pb surface with an ideal arc $c$, let $m,n$ be two positive integers with $m\mid n$, and let $\epsilon\in R$ with $\epsilon^{2m} = 1$. Then the splitting map $\Theta_{c}:S_n(\Sigma,\epsilon)\rightarrow S_n(\text{Cut}_c(\Sigma),\epsilon)$ is injective.
\end{theorem}
\begin{proof}
	Lemmas \ref{lll} and \ref{key}.
\end{proof}

\begin{theorem}
	Let $\MN$ be a connected marked  3-manifold with $\cN\neq\emptyset$,  let $D$ be a properly embedded disk in $M$, and let $e$ be an open oriented interval in $D$. Suppose $\epsilon\in R$ such that $\epsilon^4 = 1$. We have that
	$\Theta_{(D,c)}:S_n(\M,\epsilon)\rightarrow S_n( \text{Cut}_{(D,e)}\MN,\epsilon$) is injective.
\end{theorem}
\begin{proof}
	Theorem \ref{inj} and Theorem 6.7 in \cite{wang2023stated}.
\end{proof}

\begin{lemma}\label{lmis}
	Let $\MN$ be a marked 3-manifold, and let $(M,\mathcal{N}')$ be another marked 3-manifold obtained from $\MN$ by adding one extra marking $e$. Then the adding marking map $l_e: S_n(\M,-1)\rightarrow S_n(M,\cN',-1)$ is injective.
\end{lemma}
\begin{proof}
	From Corollary \ref{minus}, we have the following commutative diagram:
	\begin{equation*}
		\begin{tikzcd}
			S_n(\M,1)  \arrow[r, "l_e"]
			\arrow[d, "\Psi_n"]  
			&  S_n(M,\mathcal{N}\cup e,1) \arrow[d, "\Psi_n"] \\
			S_n(\M,-1)  \arrow[r, "l_e"] 
			&  S_n(M,\mathcal{N}\cup e,-1)\\
		\end{tikzcd},
	\end{equation*}
	where $\Psi_n$ is the linear isomorphism constructed in Theorem \ref{5.1}. Then $$l_e: S_n(\M,-1)\rightarrow S_n(M,\cN',-1)$$ is injective because $l_e: S_n(\M,1)\rightarrow S_n(M,\cN',1)$ is injective, Corollary 6.6 in \cite{wang2023stated}.
\end{proof}

\begin{theorem}
	Let $\MN$ be marked  3-manifold,  let $D$ be a properly embedded disk in $M$, and let $e$ be an open oriented interval in $D$.
	We have 
	$\Theta_{(D,c)}:S_n(\M,-1)\rightarrow S_n( \text{Cut}_{(D,e)}\MN,-1$) is injective.
\end{theorem}
\begin{proof}
	Lemma \ref{lmis} and Theorem 6.7 in \cite{wang2023stated}.
\end{proof}

\begin{conjecture}\label{Conj}
	Let $\MN$ be marked  3-manifold,  let $D$ be a properly embedded disk in $M$, and let $e$ be an open oriented interval in $D$. Suppose $\epsilon\in R$ such that $\epsilon^4 = 1$. We have that
	$\Theta_{(D,c)}:S_n(\M,\epsilon)\rightarrow S_n( \text{Cut}_{(D,e)}\MN,\epsilon$) is injective.
\end{conjecture}

We have the following confirmative examples for Conjecture \ref{Conj}:
(1) $\epsilon = \pm 1$ (2) $\epsilon$ is a primitive 4-th root of unity and every component of $M$ contains at least marking. Thus, to prove Conjecture \ref{Conj}, it suffices to show that the adding marking map is injective when adding a single marking to a 3-manifold without any existing markings, assuming \( \epsilon \) is a primitive fourth root of unity.  

\bibliographystyle{plain}

\bibliography{ref.bib}

\begin{thebibliography}{10}

\bibitem{barrett1999skein}
John~W Barrett.
\newblock Skein spaces and spin structures.
\newblock In {\em Mathematical Proceedings of the Cambridge Philosophical
  Society}, volume 126, pages 267--275. Cambridge University Press, 1999.

\bibitem{bonahon2011quantum}
Francis Bonahon and Helen Wong.
\newblock {Quantum traces for representations of surface groups in ${\rm
  SL}_2(\mathbb C)$}.
\newblock {\em Geometry \& Topology}, 15(3):1569--1615, 2011.

\bibitem{representation2}
Francis Bonahon and Helen Wong.
\newblock {Representations of the Kauffman bracket skein algebra, II: Punctured
  surfaces}.
\newblock {\em Algebraic \& geometric topology}, 17(6):3399--3434, 2017.

\bibitem{CL2022TQFT}
Francesco Costantino and Thang~TQ L{\^e}.
\newblock {Stated skein modules of 3-manifolds and TQFT}.
\newblock {\em arXiv preprint arXiv:2206.10906}, 2022.

\bibitem{detcherry2021infinite}
Renaud Detcherry.
\newblock {Infinite families of hyperbolic 3-manifolds with finite-dimensional
  skein modules}.
\newblock {\em Journal of the London Mathematical Society}, 103(4):1363--1376,
  2021.

\bibitem{unicity}
Charles Frohman, Joanna Kania-Bartoszynska, and Thang L{\^e}.
\newblock {Unicity for representations of the Kauffman bracket skein algebra}.
\newblock {\em Inventiones mathematicae}, 215:609--650, 2019.

\bibitem{frohman20223}
Charles Frohman and Adam~S Sikora.
\newblock {$SU (3)$-skein algebras and webs on surfaces}.
\newblock {\em Mathematische Zeitschrift}, 300(1):33--56, 2022.

\bibitem{gunningham2023finiteness}
Sam Gunningham, David Jordan, and Pavel Safronov.
\newblock {The finiteness conjecture for skein modules}.
\newblock {\em Inventiones mathematicae}, 232(1):301--363, 2023.

\bibitem{higgins2020triangular}
Vijay Higgins.
\newblock {Triangular decomposition of ${\rm SL}_3$ skein algebras}.
\newblock {\em Quantum Topology}, 14(1), 2023.

\bibitem{KS}
Anatoli Klimyk and Konrad Schm{\"u}dgen.
\newblock {\em Quantum groups and their representations}.
\newblock Springer Science \& Business Media, 2012.

\bibitem{le2018triangular}
Thang~TQ L{\^e}.
\newblock Triangular decomposition of skein algebras.
\newblock {\em Quantum Topology}, 9(3):591--632, 2018.

\bibitem{le2021stated}
Thang~TQ L{\^e} and Adam~S Sikora.
\newblock {Stated ${\rm SL}(n)$-skein modules and algebras}.
\newblock {\em Journal of Topology}, 17(3):e12350, 2024.

\bibitem{leY}
Thang~TQ L{\^e} and Tao Yu.
\newblock {Quantum traces for $SL_n$-skein algebras}.
\newblock {\em arXiv preprint arXiv:2303.08082}, 2023.

\bibitem{S2001SLn}
Adam Sikora.
\newblock {$SL_n$-character varieties as spaces of graphs}.
\newblock {\em Transactions of the American Mathematical Society},
  353(7):2773--2804, 2001.

\bibitem{sikora2005skein}
Adam~S Sikora.
\newblock {Skein theory for $SU(n)$-quantum invariants}.
\newblock {\em Algebraic \& Geometric Topology}, 5(3):865--897, 2005.

\bibitem{wang2023stated}
Zhihao Wang.
\newblock {On stated $ SL (n) $-skein modules}.
\newblock {\em arXiv preprint arXiv:2307.10288}, 2023.

\bibitem{wang2023representation}
Zhihao Wang.
\newblock Representation-reduced stated skein modules and algebras.
\newblock {\em Journal of Algebra}, 661:831--852, 2025.

\bibitem{yu2023center}
Tao Yu.
\newblock Center of the stated skein algebra.
\newblock {\em arXiv preprint arXiv:2309.14713}, 2023.

\end{thebibliography}

\end{document}